\RequirePackage{fix-cm}
\documentclass[smallextended]{svjour3}       % onecolumn (second format)
\smartqed  % flush right qed marks, e.g. at end of proof

\usepackage[utf8]{inputenc}
\usepackage{microtype}
\usepackage{amsmath}
\usepackage{amsfonts}
\usepackage{amssymb}
\usepackage{mathtools}
\usepackage{indentfirst}
\usepackage{array}
\usepackage{subfigure}
\usepackage{algorithm}
\usepackage{algorithmic}
\usepackage{enumitem}
\usepackage{amsopn}
\usepackage{bm}
\usepackage{url}
\usepackage{booktabs}
\usepackage{yhmath}
\usepackage{charter}
\usepackage{color,xcolor}
\usepackage[bottom]{footmisc}

\usepackage{graphicx,epstopdf,cases,multirow} % <- Preambl
\usepackage{graphicx,cases} % <- Preambl

\usepackage{rotating}
\newcommand{\qy}[1]{\textcolor{blue}{#1}}

\newcommand{\mq}[1]{\textcolor{cyan}{#1}}
\newcommand{\h}[1]{\mathbf{#1}}

\newcommand{\tensor}[1]{\mathbf{#1}}
\newcommand{\trans}{^\mathsf{T}}

\newcommand{\sign}{\mbox{sign}}

\DeclareMathOperator*{\argmin}{arg\,min}

\begin{document}

\title{Accelerated Sparse Recovery via Gradient Descent with Nonlinear Conjugate Gradient Momentum
%\thanks{Grants or other notes about the article that should go on the front page should be placed here. General acknowledgments should be placed at the end of the article.}
}
% \subtitle{Do you have a subtitle?\\ If so, write it here}

\titlerunning{Gradient Descent with Nonlinear Conjugate Gradient Momentum}        % if too long for running head

\author{Mengqi Hu  \and
        Yifei Lou \and
        Bao Wang \and
        Ming Yan \and
        Xiu Yang \and
        Qiang Ye
}

\institute{
        M. Hu \at
        Department of Industrial and Systems Engineering, Lehigh University \\
        200 West Packer Avenue, Bethlehem, 18015, PA, USA\\
        \email{meh621@lehigh.edu}         
        \and
        Y. Lou \at
        Department of Mathematical Sciences, The University of Texas at Dallas\\
        800 W. Campbell Rd, Richardson, 75080, TX, USA\\
        \email{yifei.lou@utdallas.edu}
        \and
        B. Wang \at
        Department of Mathematics and Scientific Computing and Imaging Institute, The University of Utah\\
        72 Central Campus Dr, Salt Lake City, 84102, Utah, USA\\
        \email{bwang@sci.utah.edu}
        \and
        M. Yan \at
        School of Data Science, The Chinese University of Hong Kong, Shenzhen\\ 2001 Longxiang Blvd, Shenzhen, Guangdong China\\
        Department of Computational Mathematics, Science and Engineering and Department of Mathematics, Michigan State University\\ 428 South Shaw Lane, East Lansing, 48824, MI, USA\\
        \email{myan@msu.edu}
        \and
        X. Yang \at
        Department of Industrial and Systems Engineering, Lehigh University\\
        200 West Packer Avenue, Bethlehem, 18015, PA, USA\\
        \email{xiy518@lehigh.edu}
        \and
        Q. Ye \at
        Department of Mathematics, University of Kentucky, Lexington, 40513, Kentucky, USA\\
        \email{qye3@uky.edu}
}

\date{Received: date / Accepted: date}

\maketitle

\begin{abstract}
    This paper applies an idea of  adaptive momentum for the nonlinear conjugate gradient to accelerate optimization problems in sparse recovery. Specifically, we consider two types of minimization problems: a (single) differentiable function and the sum of a non-smooth function and a differentiable function. In the first case, we adopt a fixed step size to avoid the traditional line search and establish the convergence analysis of the proposed algorithm for a quadratic problem. This acceleration is further incorporated with  an operator splitting technique to deal with the non-smooth function in the second case. We use the convex $\ell_1$ and the nonconvex $\ell_1-\ell_2$ functionals as two case studies to demonstrate the efficiency of the proposed approaches over traditional methods. 
    
    
% % \PACS{PACS code1 \and PACS code2 \and more}
\end{abstract}

\medskip

\noindent
\textbf{Declarations} 
The MATLAB codes and datasets generated during and/or analysed during the current study  will be available under \url{https://sites.google.com/site/louyifei/Software} after publication.

\section{Introduction}
Traditional methods for reconstructing signals from measured data follow the well-known Nyquist-Shannon sampling theorem \cite{Shannon}, which guarantees the exact  recovery if 
the sampling rate is at least twice the highest frequency of the underlying signal. Similarly, the fundamental theorem of linear algebra suggests that the number of linear measurements  of a discrete finite-dimensional signal should be at least as large as its ambient dimension to ensure a stable reconstruction.
Nyquist–Shannon  theorem serves as the underlying principle of most devices \cite{boggess2015first}  such as analog-to-digital conversion, medical imaging, and video processors, %\yl{drawbacks of such system? you can also mentions some applications that involve under-determined system}
but it is a sufficient condition for the exact recovery of any signal that requires an overly large number of measurements to be collected.

% . The large sampling size that is required by the theorem raises challenges as the request for larger data transmitting and demand for higher resolution images, especially for applications like remote sensing.

To acquire and process data more economically, the paradigm of compressive sensing (CS) \cite{donoho2006compressed}–-also known as compressed sensing, or compressive sampling–-provides a fundamentally new approach that reconstructs  certain signals from what was believed in the past to be highly incomplete measurements (information). 
CS relies on an empirical observation that most signals can be well approximated by a sparse expansion under a properly chosen basis, that is, by only a small number of non-zero coefficients. 
The number of non-zero entries of a vector $\h x\in\mathbb R^N$ is denoted by
$\Vert \h x\Vert_0$.
%-------------------------------------------------------------------------------
%and the $\ell_1$ norm of $\h x$ is defined as
%-------------------------------------------------------------------------------
%$\Vert\h x\Vert_1\Def \sum_{n=1}^N |x_n|$.
%-------------------------------------------------------------------------------
Note that $\Vert\cdot\Vert_0$ is named the ``$\ell_0$ norm" in \cite{donoho2006compressed},
although it is not even a semi-norm. The vector $\h x$ is called 
\emph{$s$-sparse} if $\Vert \h x\Vert_0\leq s$, and it is considered a sparse vector if $s\ll N$. 
% Few practical systems have truly sparse gPC coefficients, but rather \textit{compressible}, i.e., only a few entries contributing significantly to its $\ell_1$ norm. 
Note that few practical systems are truly sparse through direct observations, but rather \textit{compressible}, i.e., only a few entries contribute significantly to its $\ell_1$ norm under certain transformations. 
% To this end, 
% %we define a vector $\h x_s$ as the best $s$-sparse approximation that can be obtained by setting all but the $s$-largest entries in magnitude of $\h x$ to zero. 
% a vector $\h x_s$ is defined as the best $s$-sparse approximation, obtained by setting all but the $s$-largest entries in magnitude of $\h x$ to zero; and subsequently, $\h x$ is regarded as sparse or compressible if 
% $\Vert \h x - \h x_s\Vert_1$ is small for $s\ll N$. 

For simplicity, we assume  linear measurements; otherwise one can always linearize the data collection process. Consequently, we consider a data vector $\h b\in \mathbb R^M$ obtained by
\begin{equation}
    \h b= A \h x + \h n,\label{eq:ax=b}
\end{equation}
where $ A\in \mathbb R^{M \times N}$  is called a sensing matrix,  $\h x\in \mathbb R^N$ is
an underlying signal to be recovered, and $\h n\in\mathbb R^M$ is the noise term. {We assume the noise follows i.i.d. Gaussian distribution.}
%The concept of \emph{sparsity} plays an important role in solving the linear system \eqref{eq:ax=b}.
To find a sparse vector $\h x$ from \eqref{eq:ax=b}, one formulates an unconstrained minimization problem,
%\yl{make consistent of $\argmin$ or $\argmin$ throughout the thesis}
\begin{equation}\label{eq:l0}
  \hat{\h x}_0=\argmin_{\h x}\lambda\Vert \h x\Vert_0+\frac 1 2 \Vert  A \h x - \h b\Vert_2^2,  
\end{equation}
where $\lambda$ is a positive parameter to be tuned such that $\Vert A
\hat{\h x}_0-\h b\Vert_2\leq\epsilon$ for a pre-set error tolerance $\epsilon$ that often corresponds to {the standard deviation of the random Gaussian noise. For other types of noise, e.g., Poisson noise, then the least-squares formulation  $\|A\h x-\h b\|_2^2$ is not a good choice of the data misfit.}
As the $\ell_0$ minimization \eqref{eq:l0} is NP-hard~\cite{natarajan95}, one  replaces it by the convex  $\ell_1$ norm, i.e.,
\begin{equation}\label{eq:l1}
  \hat{\h x}_1=\argmin_{\h x}\lambda\Vert \h x\Vert_1+ \frac 1 2 \Vert A \h x - \h b\Vert_2^2.  
\end{equation} 
In this paper, we consider a general formulation for sparse recovery
\begin{equation}\label{eq:proximal setup}
    \min_{\h x}  \lambda f(\h x) + g(\h x),
\end{equation}
where $f(\cdot)$ is a regularization term and $g(\cdot)$ is a (convex and differentiable) data fidelity term, e.g., $g(\h x)=\frac 1 2 \Vert A \h x - \h b\Vert_2^2$. %{Need help.} 
We assume that $f$ is a continuous (possibly non-differentiable) function
%, and $g$ is convex and differentiable. 
% \begin{equation}\label{eq:non-convex}
%   \hat{\h x}=\argmin_{\h x} \lambda f(\h x)+\frac 1 2 \Vert \h A \h x - \h b\Vert_2^2,
% \end{equation}
%with an arbitrary regularization $f$ 
that can enhance the sparsity of $\h x$. 
% There are many non-convex alternatives to approximate the $\ell_0$ norm that give
% superior results over the $\ell_1$ norm, such as $\ell_{1/2}$
% \cite{chartrand07,xuCXZ12,laiXY13}, capped $\ell_1$
% \cite{zhang2009multi,shen2012likelihood,louYX16}, transformed
% $\ell_1$~\cite{lv2009unified,zhangX17,zhangX18, GuoLLtransform18}, $\ell_1$-$\ell_2$  \cite{yinEX14,louYHX14,louY18},  $\ell_1/\ell_2$ \cite{l1dl2,L1dL2_accelerated} and error function (ERF) \cite{guo2021novel}.
For instance, the non-convex metric $\ell_{p}$ for $p\in(0,1)$ can be viewed as a continuation effort to approximate $\ell_{0}$ as $p \rightarrow 0$ {\cite{chan2014half,lanzageneralized,huang2017majorization}}. Another regularization that achieves a continuation from $\ell_0$ to $\ell_1$ is the error function (ERF) \cite{guo2021novel} by changing its internal parameter. Some non-convex regularizations derived from $\ell_{1}$ include capped $\ell_1$ \cite{zhang2009multi,shen2012likelihood,louYX16}, transformed $\ell_1$ (TL1)~\cite{lv2009unified,zhangX17,zhangX18,GuoLLtransform18}, and sorted $\ell_{1}$ \cite{huang2015nonconvex}.
A combination of different norms can also be served as a sparsity promoting sparsity, e.g.,  $\ell_1-\ell_2$  \cite{yinEX14,louYHX14,louY18} and $\ell_1/\ell_2$ \cite{l1dl2,L1dL2_accelerated}.
To the best of our knowledge, only $\ell_1-\ell_2$ and TL1 have the exact sparse recovery guarantees based on the RIP type of conditions \cite{yinLHX14,zhangX18}, which are actually more strict compared to the one for the $\ell_1$ model. As these RIP conditions are sufficient and unverifiable, many works reported the empirical advantages of non-convex regularizations over the convex $\ell_1$ approach in promoting sparsity. {A major difficulty in minimizing \eqref{eq:proximal setup} for a nonconvex regularization $f(\cdot)$ is that many algorithms may be stuck at the local optimal solutions.}
%when solving for  the inverse problem from  \eqref{eq:ax=b}. 

As $f(\cdot)$ is non-differentiable, gradient-based optimization methods can not be directly applied to 
 minimize \eqref{eq:proximal setup}, not to mention some acceleration techniques by adaptive momentum \cite{POLYAK19641,polak1969note,hestenes1952methods,fletcher1964function}. One remedy involves a smooth approximation of $f$ such as using the  Huber function~\cite{huber1987place,hermey1999fitting,sun2020adaptive} to approximate the $\ell_1$ norm. In general, several papers reported using  smoothing  to approximate non-smooth functions to improve the performance of non-linear conjugate algorithms~\cite{chen2010smoothing,wu2019signal,pang2016smoothing,narushima2013smoothing}.
 %The Huber loss  is commonly used in regression analysis that combines the benefits of  small error  estimation from the $\ell_2$ penalty and robustness against outliers from the $\ell_1$. %penalty. Later, it is used to deal with the non-smooth $\ell_1$ penalty in sparse recovery since it resembles some properties of the $\ell_1$ penalty \cite{huber1987place,hermey1999fitting,sun2020adaptive}.
Another alternative is based on operator splitting to deal with the non-smooth term $f(
\cdot)$ and the smooth function $g(\cdot)$ separately, for example, 
 forward-backward splitting (FBS) \cite{combettes2005signal}, the alternative direction method of multipliers (ADMM) \cite{boydPCPE11admm}, and iteratively reweighted $L_1$ \cite{candes2008enhancing,lu2014iterative}.

We propose to combine the operator splitting with the momentum acceleration. In particular, we incorporate the momentum update in the gradient descent when minimizing the data fitting term $g(\cdot)$ for speed-up, while relying on proximal operators \cite{parikh2014proximal} to deal with the non-differentiable function $f(\cdot)$. Starting by $f(\cdot)=\emptyset,$ i.e., minimizing a single differentiable function $g(\cdot)$, we promote the choice of fixed step size in the momentum-based gradient descent algorithm and analyze its convergence rate for a quadratic problem. To deal with the non-smooth function $f(\cdot)$, we further adopt a splitting technique and consider two case studies when the proximal operator according to $f(\cdot)$ has a closed-form solution. We conduct experiments on a quadratic problem, $\ell_1$ and $\ell_1-\ell_2$ minimization problems to compare among different momentum update formulas and showcase the speed-up of the proposed approach  with simple implementation over the traditional gradient-based approaches.

The remaining of this paper is organized as follows. Section~\ref{sect:opt review} examines the case of minimizing a single differentiable function. In particular, we advocate a constant step size and prove the convergence for a quadratic problem. 
The proposed marriage of FBS and momentum acceleration is discussed in Section~\ref{sect:min-2} with  experiments on two case studies of $\ell_1$ and $\ell_1-\ell_2$ regularizations, showing the faster convergence of the proposed method than existing approaches. Finally, conclusions and future works are presented in Section~\ref{sect:conclude}.

\section{Minimizing a single function}\label{sect:opt review}
We review  in Section~\ref{sect:grad} gradient-based algorithms that minimize a single function, including gradient descent, conjugate gradient, and adaptive momentum methods. We propose to  combine the Fletcher-Reeve moment and gradient descent with a fixed step size in Section~\ref{sect:prop-alg}. The convergence of the proposed scheme can be established  for a quadratic problem. Lastly, experimental comparison is presented in Section~\ref{sect:exp_quad}. 

\subsection{Literature review}\label{sect:grad}
Gradient descent is a class of first-order iterative optimization algorithms for finding a local minimum of a differentiable function. This type of algorithms involves repeated moving along the opposite direction of the gradient of the objective function at the current point,  since it is the direction where function value decreases at the fastest rate. 

Given a differentiable function $g(\cdot)$, a general form of gradient descent (GD) that  minimizes $g(\tensor{x})$ can be described as, 
\begin{align}
    \begin{cases}
        \h p^{(l+1)} &= -\nabla g(\h x^{(l)})\\
        \h x^{(l+1)} &= \h x^{(l)} + \alpha^{(l+1)}\h p^{(l+1)},
    \end{cases}\label{eq:gradient descend}
\end{align}
%The step size $\alpha$ can be a constant which do not require any update.
where $l$ indexes the iteration number and
$\alpha^{(l+1)}>0$ is a step size that can be fixed or updated iteratively.
There are many variations of GD depending on how the step size is determined and/or the descending direction is chosen. 
For example,   steepest descent (SD) is perhaps one of the simplest variations, 
which goes as follows,
\begin{align}
    \begin{cases}
        \h p^{(l+1)} &= -\nabla g(\h x^{(l)})\\
        \alpha^{(l+1)} &= \argmin\limits_{\alpha } g(\h x^{(l)} + \alpha\h p^{(l+1)})\\
        \h x^{(l+1)} &= \h x^{(l)} + \alpha^{(l+1)}\h p^{(l+1)}.
    \end{cases}\label{eq:steepest descend}
\end{align}
In each iteration, SD performs an exact line search  to achieve the maximum descent along the gradient direction, i.e., the descent is the steepest. However, empirically it does not work  well in most cases, since such a local descending property does not necessarily coincide with the overall descending of the original function.

Notice that the search direction in each iteration of \eqref{eq:steepest descend} only utilizes the information at the current step $\h x^{(l)}$ without any  information from previous iterations. Adding them back  leads to  momentum-based algorithms, which are also called as heavy ball algorithms \cite{POLYAK19641}. The term ``momentum'' is an analogy of a heavy ball sliding on the surface of values of the function being minimized when the update of each step is memorized in the process.
To this end, we refer the following iteration
\begin{align}
    \begin{cases}
    \h p^{(l+1)} &= -\nabla g(\h x^{(l)}) + \beta^{(l+1)}\h p^{(l)} \\
    \h x^{(l+1)} &= \h x^{(l)} + \alpha^{(l+1)}\h p^{(l+1)},
    \end{cases}\label{eq: grad momentum}
\end{align}
as gradient descent with momentum (GDM).  Both   $\alpha^{(l+1)}$ and $\beta^{(l+1)}$ in \eqref{eq: grad momentum} can be fixed or adaptively chosen according to  a certain scheme. For instance, if we update  $\alpha^{(l+1)}$ in the same way as SD \eqref{eq:steepest descend}, the corresponding algorithm 
\begin{align}
    \begin{cases}
 \beta^{(l+1)} &= \frac{\|\nabla g(\h x^{(l)})\|^2}{\|\nabla g(\h x^{(l-1)})\|^2}\\
        \h p^{(l+1)} &= -\nabla g\left(\h x^{(l)} \right)  + \beta^{(l+1)} \h p^{(l)}\\
         \alpha^{(l+1)} &= \argmin\limits_{\alpha } g(\h x^{(l)} + \alpha\h p^{(l+1)})\\
        \h x^{(l+1)} &= \h x^{(l)} + \alpha^{(l+1)}\h p^{(l+1)},
        \end{cases}\label{eq:CG FR2}
\end{align}
is identical to the classic nonlinear conjugate gradient (CG).
The $\beta$ update for momentum coefficient is called Fletcher-Reeves (FR) momentum in nonlinear conjugate gradient algorithms  \cite{hestenes1952methods,fletcher1964function}.
 In addition to FR, other popular momentum updates include 
\begin{itemize}
	\item Polak-Ribi\`{e}re (PR)  \cite{polak1969note}
	\begin{equation}\label{eq:CG PR}
	 \beta_{PR}^{(l+1)} = \frac{\langle\nabla g(\h x^{(l)}),\nabla g(\h x^{(l)})-\nabla g(\h x^{(l-1)})\rangle}{\|\nabla g(\h x^{(l-1)})\|^2};
	\end{equation}
	\item Hestenes-Stiefel (HS) \cite{Hestenes&Stiefel:1952}
	\begin{equation}\label{eq:CG HS}
	 \beta_{HS}^{(l+1)} = \frac{\langle\nabla g(\h x^{(l)}),\nabla g(\h x^{(l)})-\nabla g(\h x^{(l-1)})\rangle}{-\langle{\h p}^{(l)},\nabla g(\h x^{(l)})-\nabla g(\h x^{(l-1)})\rangle};
	\end{equation}
	\item Dai-Yuan (DY) \cite{dai1999nonlinear}
	\begin{equation}\label{eq:CG DY}
	\beta_{DY}^{(l+1)} = \frac{\|\nabla g(\h x^{(l)})\|^2}{-\langle {\h p}^{(l)},\nabla g(\h x^{(l)})-\nabla g(\h x^{(l-1)})\rangle}.
	\end{equation}
\end{itemize}

%\qy{Question: should $ \h x^{(l)}$ in the denominator in HS and DY be the search direction $\h p^{(l)}$?  }

Another type of momentum-based algorithms was developed by Yurii Nesterov \cite{nesterov1983method,nesterov2003introductory}. Starting from $t^{(0)} = 1$, Nesterov's accelerated gradient (NAG) is expressed as, 
\begin{align}
    \begin{cases}
    t^{(l+1)} &= \frac{ 1+\sqrt{4(t^{(l)})^2 + 1} }{2},\\
    % \h y^{(l+1)} &= \h x^{(l)} + \frac{t^{(l)} - 1}{t^{(l+1)}}\h p^{(l)}\\
    \h p^{(l+1)} &= -\nabla g\left( \h x^{(l)} \right),\\
    \h y^{(l+1)} &= \h x^{(l)} + \alpha^{(l+1)}\h p^{(l+1)},\\
    \h x^{(l+1)} &= \h y^{(l+1)} + \frac{t^{(l)} - 1}{t^{(l+1)}}(\h y^{(l+1)} - \h y^{(l)}).
    \end{cases}\label{eq:Nest accel grad}
\end{align}
Similarly to other gradient-based algorithms, the step size $\alpha^{(l+1)}$ in NAG can be fixed or updated during the iteration. For convex function $g(\cdot),$ NAG  achieves a convergence rate of $O( \frac{1}{l^2})$, as opposed to $O(\frac 1 l)$ obtained by
standard gradient-based methods. 
This momentum scheme can be further accelerated by a proper restart with provable guarantees in certain circumstances \cite{nemirovski1985optimal,giselsson2014monotonicity,su2014differential,roulet2020sharpness}.

\subsection{{FR momentum gradient descent}}\label{sect:prop-alg}

% \qy{Question: do we want to use a more descriptive section title such as FR momentum Gradient descent?}

Exact line search is not necessarily optimal, as the gradient descent direction may not be a good search direction. As a result, the steepest descent algorithm \eqref{eq:steepest descend} bounces back and forth in the valley formed by the objective function rather than down the valley.  Similar conclusion can be drawn for certain momentum based algorithms.
As pointed out in \cite{powell1977restart,hager2006survey},  exact line search would lead to a very small step size in such a way that two consecutive iterates do not vary too much; this phenomenon is called \textit{jamming}. To avoid jamming, one can use  a hybrid momentum scheme \cite{dai2001efficient,andrei2008another,nocedal2006numerical} or an inexact line search \cite{hager2005new,rivaie2015new}. Instead of exact search as used in SD, an inexact search refers to finding a step size $\alpha$ that satisfies the Wolfe conditions \cite{al1985descent,dai1999nonlinear,gilbert1992global,hager2005new}, i.e.,
\begin{align}
    \begin{cases}
        g(\h x^{(l)} + \alpha^{(l)} \h p^{(l)}) \leqslant g(\h x^{(l)} ) + c_{1}  \langle \alpha^{(l)}\h p^{(l)}, \nabla g(\h x^{(l)}) \rangle,\\
          \langle -\h p^{(l)}, \nabla g(\h x^{(l)} + \alpha^{(l)} \h p^{(l)})\rangle \leqslant  c_{2} \langle -\h p^{(l)}, \nabla g(\h x^{(l)})\rangle,
    \end{cases}\label{eq:wolfe}
\end{align}
for two constants $ 0< c_{1} < c_{2} < 1$. The first equation of \eqref{eq:wolfe}  is also called Armijo-Goldenstein condition 
\cite{armijo1966minimization,Bertsekas99}, which is usually used in back-tracking step sizes. These conditions play an important role in establishing a descent property and global convergence of conjugate descent. 
Instead of designing a update scheme for $\alpha,$ we consider a fixed step size $\alpha$ in the gradient descent that is combined with the FR moment, i.e.,
\begin{align}
    \begin{cases}
 \beta^{(l+1)} &= \frac{\|\nabla g(\h x^{(l)})\|^2}{\|\nabla g(\h x^{(l-1)})\|^2}\\
        \h p^{(l+1)} &= -\nabla g\left(\h x^{(l)} \right)  + \beta^{(l+1)} \h p^{(l)}\\
        \h x^{(l+1)} &= \h x^{(l)} + \alpha\h p^{(l+1)},
        \end{cases}\label{eq:CG FR}
\end{align}
which is referred to as FR gradient descent (FRGD). By fixing $\alpha$, most properties  used in the convergence  analysis of conjugate gradient no longer hold.
Fortunately, we can borrow a technique used in an inexact conjugate gradient (due to round off errors) or inexact preconditioning \cite{golub1999inexact,tong2000analysis} to analyze the convergence of FRGD \eqref{eq:CG FR}. Theorem \ref{thm:quad conv rate} characterizes the convergence analysis of the proposed FRGD for  a  quadratic  problem, 
\begin{align}
    \min_{\h x} g(\h x) = \frac{1}{2} \h x\trans A \h x + \h x\trans\mathbf{b},\label{eq:quadratic}
\end{align}
where $A$ is a strictly symmetric positive definite matrix. To this end, we define the condition number of $A$ as $\kappa(A) = |\lambda_{\max}(A)/\lambda_{\min}(A)|$, i.e., the ratio between the largest and smallest eigenvalues.

\begin{theorem}\label{thm:quad conv rate}
Suppose $\{\h x^{(l)}, \h p^{(l)}\}$ be generated by  \eqref{eq:CG FR} with a fixed step size $\alpha$ when minimizing \eqref{eq:quadratic}.  Let $\h r^{(l)} = \nabla g(\h x^{(l)})$, $\rho = \max_{0 \leqslant j \leqslant i \leqslant l-1} \|\h r^{(i)}\|_2/\| \h r^{(j)} \|_2$, $\h z^{(l)} = \h r^{(l)} / \|\h r^{(l)} \|_2$ and $Z^{(l)} = [\h z^{(0)}, \h z^{(1)},\cdots, \h z^{(l-1)}]$. If $\h z^{(0)}, \h z^{(1)},\cdots, \h z^{(l)}$ are linearly independent, then there exists a constant \begin{align} K_{l} \leqslant l(1+\frac{l\rho} 2)\|A\|_{2}\kappa(Z^{(l+1)}),\end{align}
such that
\begin{align}\label{eq:conv-rate}
    \|\h r^{(l)}\|_2 \leqslant 2(1 + K_{l}) \left( \frac{\sqrt{\kappa(A)} - 1}{\sqrt{\kappa(A)} + 1} \right)^{l} \|\h r^{(0)}\|_2.
\end{align}

\end{theorem}

\begin{proof}
   Denote $R^{(l)} = [\h r^{(0)}, \cdots, \h r^{(l-1)}]$ and $D^{(l)} = \mathrm{diag}\{\|\h r^{(0)}\|_2, \cdots, \|\h r^{(l-1)}\|_2\}$, then $Z^{(l)} = R^{(l)} (D^{(l)})^{-1}$.
   We further denote $P^{(l)} = [\h p^{(0)}, \cdots, \h p^{(l-1)}]$. It follows from {$\h r^{(l+1)} = \h r^{(l)} - \alpha A p^{(l)}$} that  %{ double checking}
    \begin{align}
        \alpha A P^{(l)} &= [\h r^{(0)} - \h r^{(1)}, \cdots, \h r^{(l-1)} - \h r^{(l)}]\notag\\
            &= R^{(l)} L^{(l)} - \h r^{(l)}\h e_{(l)}\trans, \notag\\
            &= {{\h Z}}^{(l)} D^{(l)} L^{(l)} - \h r^{(l)}\h e_{(l)}\trans ,\label{eq:p decomp}
    \end{align}
    where $\h e_{(l)} = [0,\cdots,0,1]\trans$ is an {$l \times 1$ column vector} and $ L^{(l)}$ is the $l \times l$ lower bidiagonal matrix with $1$ on the diagonal and $-1$ on the subdiagonal.
    Similarly using the $\h p$ update of $\h p^{(l)} = \h r^{(l)} + \beta^{(l)} \h p^{(l-1)} $, we have 
    \begin{equation}
         Z^{(l)} = R^{(l)} (D^{(l)})^{-1} = P^{(l)} U^{(l)} (D^{(l)})^{-1}, \label{eq:z decomp}
    \end{equation}
    where  $ U^{(l)}$ is the $l \times l$ upper bidiagonal matrix with $1$ on the diagonal and $-\beta^{(1)}$, $\cdots$, $-\beta^{(l-1)}$ on the subdiagonal.
    Combining \eqref{eq:p decomp} \eqref{eq:z decomp}, we obtain
    \begin{align}
        A  Z^{(l)} =  Z^{(l)} T^{(l)} - \frac{\h r^{(l)} \h e_{(l)}\trans}{\hat\alpha \|\h r^{(0)}\|},
    \end{align}
    where $T^{(l)} = \frac{1}{\alpha} D^{(l)}L^{(l)}U^{(l)}(D^{(l)})^{-1}$ and $\hat\alpha= \alpha \|\h r^{(l-1)}/ \|\h r^{(0)}\|$. It is straightforward to verify that $\hat\alpha = \h e_{(l)}\trans (T^{(l)})^{-1} \h e_{(1)}$. 
    By  \cite[Theorem 3.5]{tong2000analysis} we have
    \begin{align}
        \|\h r^{(l)} \|_2 \leqslant (1+K_{l}) \min_{p\in\mathcal{P}_{l}, p(0) = 1} \|p(A) \h r^{(0)}\|_2,
    \end{align}
    where $K_{l} = \|A Z^{(l)} T^{(l)} [I_{(l)}, 0] Z^{(l+1)}_{\dagger}\|_{2} \leqslant  \|A\|_{2} \|T^{(l)}\|_{2} \| Z^{(l)} \|_{2} \| Z^{(l+1)}_{\dagger}\|_{2}$,  $Z^{(l+1)}_{\dagger}$ is the pseudo-inverse of $Z^{(l+1)}$, and $\mathcal{P}_{l}$ is the space of polynomials of degree $l$. 
  By definition  $\beta^{(l)} = \|\h r^{(l)}\|_2^{2} / \|\h r^{(l-1)}\|_2^{2},$ we can rewrite 
    \begin{align}
        T^{(l)} = \frac{1}{\alpha} D^{(l)}L^{(l)} (D^{(l)})^{-1} D^{(l)}U^{(l)}(D^{(l)})^{-1}
        &= \frac{1}{\alpha} \Tilde{L}^{(l)} \Tilde{U}^{(l)},
    \end{align}
    where
    \begin{align}
        \Tilde{L}^{(l)} &= D^{(l)}L^{(l)} (D^{(l)})^{-1}\nonumber\\
        &= \begin{bmatrix}
            1   &   &   &   &\\
            -\frac{\|\h r^{(1)}\|}{\|\h r^{(0)}\|}  & 1 & & &\\
            &   -\frac{\|\h r^{(2)}\|}{\|\h r^{(1)}\|}  & 1 & &\\
            & & \ddots &\ddots &\\
            & & &-\frac{\|\h r^{(l-1)}\|}{\|\h r^{(l-2)}\|}  & 1
        \end{bmatrix}\\ \notag
        &= \begin{bmatrix}
            1   &   &   &   &\\
            -\sqrt{\beta^{(1)}}  & 1 & & &\\
            &    -\sqrt{\beta^{(2)}}  & 1 & &\\
            & & \ddots &\ddots &\\
            & & & -\sqrt{\beta^{(l-1)}}  & 1
        \end{bmatrix},\notag
    \end{align}
    and
    \begin{align}
        \Tilde{U}^{(l)} &= D^{(l)}U^{(l)} (D^{(l)})^{-1}\notag\\
        &= \begin{bmatrix}
            1   & -\beta^{(1)}\frac{\|\h r^{(0)}\|}{\|\h r^{(1)}\|}  &   &   &\\
              & 1 & -\beta^{(2)}\frac{\|\h r^{(1)}}{\h r^{(2)}\|}& &\\
            & & \ddots &\ddots &\\
            &     &  & 1 &-\beta^{(l-1)}\frac{\|\h r^{(l-2)}\|}{\|\h r^{(l-1)}\|} \\
            & & & & 1
        \end{bmatrix}\notag\\
        &= \begin{bmatrix}
            1   & -\sqrt{\beta^{(1)}}  &   &   &\\
              & 1 &-\sqrt{\beta^{(2)}}  & &\\
            & & \ddots &\ddots &\\
            &     & & 1 &-\sqrt{\beta^{(l-1)}}\\
            & & &   & 1
        \end{bmatrix}= (\Tilde{L}^{(l)})\trans.
    \end{align}
    We can see that $(\Tilde{L}^{(l)})^{-1}$ is a lower triangular matrix with 1 on its diagonal and $(i,j)$-th entry be $\sqrt{\beta^{(j)}\beta^{(j+1)}\cdots\beta^{(i)}} = \|\h r^{(i-1)}\|_{2}/\|\h r^{(j-1)}\|_{2}$ for all $i>j$. Let $\rho$ be an upper bound of $\|\h r^{(i-1)}\|_{2}/\|\h r^{(j-1)}\|_{2}$, then we have $\|(\Tilde{L}^{(l)})^{-1}\|^{2}_{F} \leqslant l + l(l-1)\rho/2$ and therefore 
    \begin{align}
        \|(T^{(l)})^{-1}\|_{2} &= \alpha \|(\Tilde{L}^{(l)}(\Tilde{L}^{(l)})\trans)^{-1}\|_{2} =  \alpha \|(\Tilde{L}^{(l)})^{-1}\|_{2}^{2}\\
        &\leqslant \alpha\|(\Tilde{L}^{(l)})^{-1}\|^{2}_{F} 
        \leqslant \alpha l( 1 + l\rho/2).
    \end{align}
    Combining with the fact $\| Z^{(l)}\|_{2}\| Z^{(l+1)}_{\dagger}\|_{2}\leqslant \| Z^{(l+1)}\|_{2}\|Z^{(l+1)}\|_{2} = \kappa( U^{(l+1)})$ 
    % \yl{what's u?} 
    yields $K_{l} \leqslant l\alpha(1+l\rho/2)\|A\|_{2}\kappa(\h Z^{(l+1)})$. 
    Finally, it follows the standard conjugate gradient convergence bound \cite{saad2003iterative} that gives
    \begin{align*}
        \min_{p\in\mathcal{P}_{l}, p(0) = 1} \|p(A) \h r^{(0)}\| _{2}
       & \leqslant \min_{p\in\mathcal{P}_{l}, p(0) = 1} \max_{i} |p(\lambda_{i})|\|\h r^{(0)}\|_{2}\\
       & \leqslant 2\left( \frac{\sqrt{\kappa(A)} - 1}{\sqrt{\kappa(A)} + 1} \right)^{l} \|\h r^{(0)}\|_{2},
    \end{align*}
    where $\lambda_{i}$ are the eigenvalues of matrix $A$ and the result follows. \qed
\end{proof}
We see that this convergence rate \eqref{eq:conv-rate} is similar to the one of the classic conjugate gradient algorithm, which is given by
\begin{align}\label{eq:conv-rateCG}
    \|\h r^{(l)}\|_{\qy{A^{-1}}} \leqslant 2 \left( \frac{\sqrt{\kappa(A)} - 1}{\sqrt{\kappa(A)} + 1} \right)^{l} \|\h r^{(0)}\|_{\qy{A^{-1}}}.
\end{align}
The difference between \eqref{eq:conv-rate} and \eqref{eq:conv-rateCG} lies in the additional term of $1 + K_{l}$ introduced by the fixed step size. Similar to other accelerated gradient schemes and Krylov subspace methods, this convergence rate of \eqref{eq:conv-rate} and \eqref{eq:conv-rateCG} is achieved only when the current position is not in the neighborhood of a stationary point. \mq{In other words, the convergence is only achieved at the first few iterations, not when the iteration number goes to infinity. We will demonstrate such a phenomenon later in the numerical examples.}
% {needs working}As the iteration increases, especially when the number of iteration exceeds to the dimension of the problem, the linear independence of the columns in the matrix $Z$ is no longer guaranteed, and hence the acceleration no longer works. 
In practice, the vectors $\h p,A\h p,\cdots A^{k}\h p$ usually form an ill-conditioned basis for the Krylov subspace.
%. This fact is especially clear for FR method since all the basis vector i.e. search directions, are determined by a classic Gram-Schmidt process, and classic Gram-Schmidt are well known for being numerically unstable generating othonomal basis. When look at the vectors $p,Ap,\cdots A^{k}p$ that spans the Krylov subspace, one may recall the classical power method. Under certain conditions, this
As $k\rightarrow \infty$, $A^k\h p$ points to nearly the same direction, which is the  eigenvector corresponding to the dominant eigenvalue of $A$ according to the power method. As a result, the acceleration is less effective. This is a common drawback for all Krylov subspace based algorithms \cite{fundamentals_Sec6.2,liesen2013krylov2.3} such as Arnoldi, Lanczos, conjugate gradient, etc. There are methods designed to address this issue, which falls outside our paper's scope.

{Theorem~\ref{thm:quad conv rate} is based on the structure of momentum iterations, which is applicable to other formulations of momentum such as  PR  (\ref{eq:CG PR}), HS (\ref{eq:CG HS}), and  DY (\ref{eq:CG DY}). The effect of these different formulations of momentum is on the quality of the basis  $\h z^{(0)}, \h z^{(1)},\cdots, \h z^{(l)}$ constructed, which is more complicated to analyze and will be left as future work. }

{We observe empirically  that the FR formulation of $\beta^{(l+1)} $ is most stable among the four. Note that all  four formulations are equivalent for a quadratic function $\frac{1}{2} \h x\trans A \h x + \h x\trans\mathbf{b}$ and they all enforce the $A$-orthogonality of $\{\h p^{(l)}\}$ and hence the orthogonality of $\{\h z^{(l)}\}$ when exact search $$\alpha^{(l+1)} = \argmin\limits_{\alpha } g(\h x^{(l)} + \alpha\h p^{(l+1)}),$$ is used. However,
%the FR and PR formulations are defined using the gradients $\nabla g\left(\h x^{(l)} \right) $, while 
the HS and DY formulations use the search direction $\h p^{(l)}$ in addition to the gradient. As a result, the perturbation on  $\beta^{(l+1)}$ for HS and DY is affected through %the gradient only for FR and PR formulations
%, but for the HS and DY formulations, it will be through 
both the gradient and the search direction, when using a fixed step size $\alpha$. We observe empirically that the range of the step size required by HS and DY  to converge is usually  one to two orders of magnitude smaller than the one for FR and PR.
Thus, the FR version, having the simplest formulation, has the least first-order perturbation errors and can be expected to be more stable.} 

\subsection{Experimental results}\label{sect:exp_quad}
{
\subsubsection{Rosenbrock function}
We start with a textbook example of
the Rosenbrock function defined by $$f(x,y) = 100 (y -x^2)^2 + (1-x)^2,$$ 
as a test problem. This function has a long parabolic-shaped flat valley, depicted in Figure~\ref{fig:rosen}, which causes difficulty for any algorithm to converge to the global minimum of $f(x,y)$.}
\begin{figure}
    \centering
    \subfigure[FR]{
    \includegraphics[width = 0.4\textwidth]{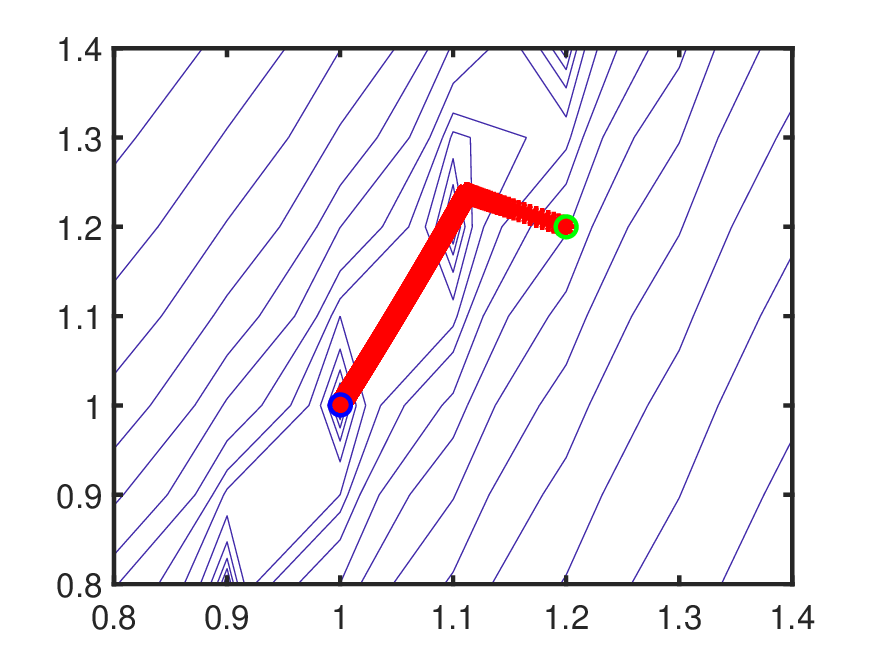}
    }\quad    
    \subfigure[DY]{
    \includegraphics[width = 0.4\textwidth]{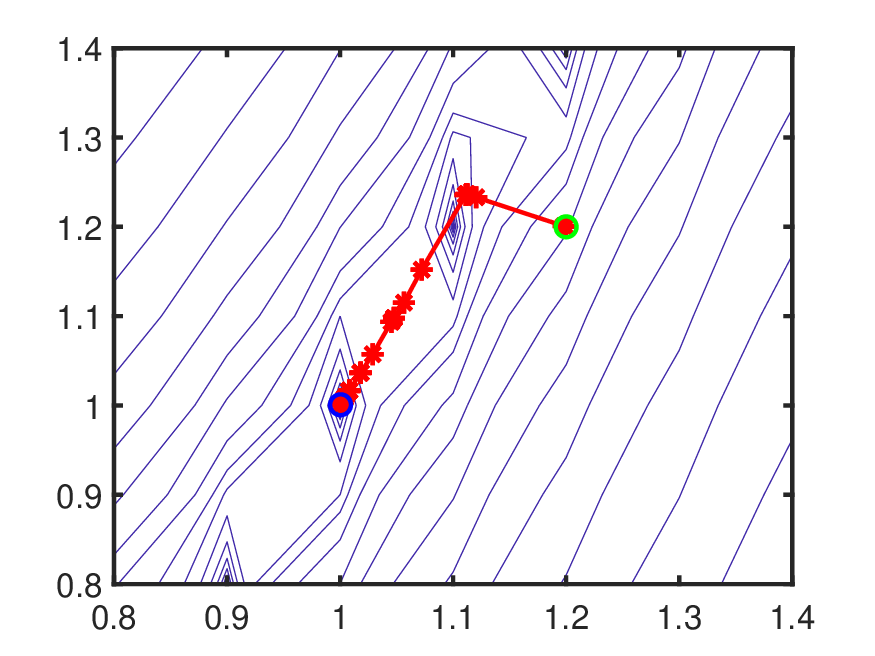}
    }\\
    \subfigure[PR]{
    \includegraphics[width = 0.4\textwidth]{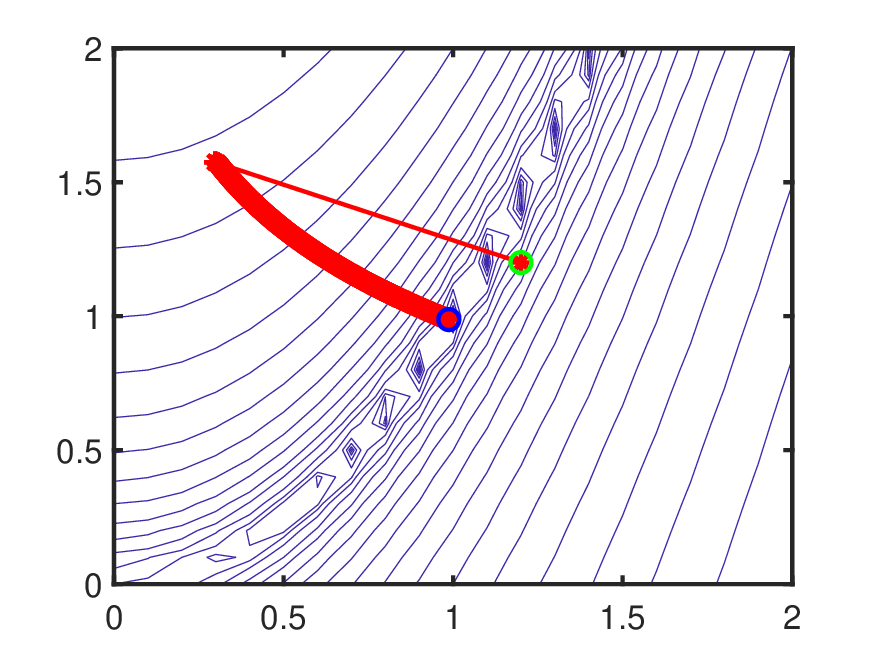}
    }\quad
    \subfigure[HS]{
    \includegraphics[width = 0.4\textwidth]{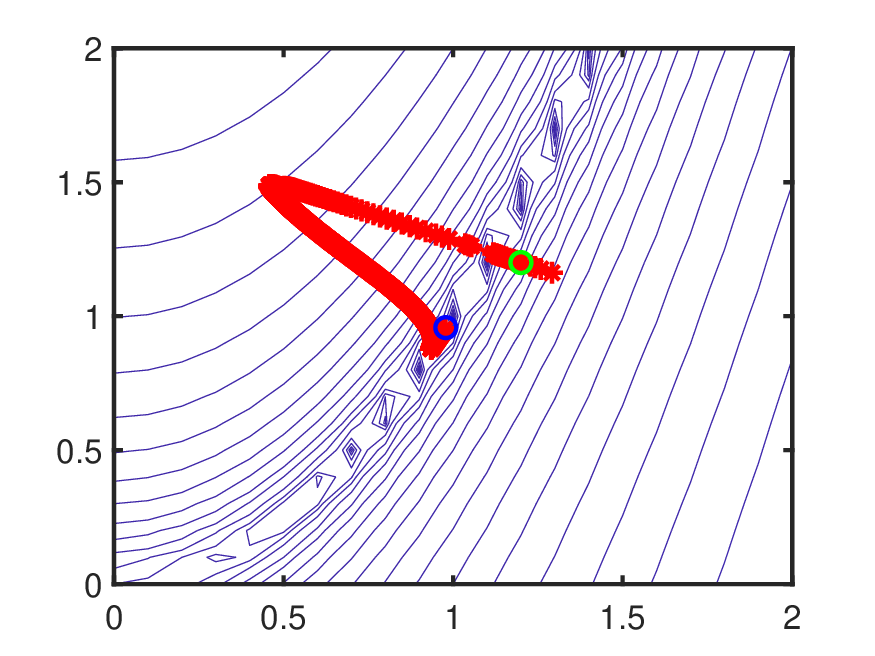}
    }
    \caption{Iterative solutions (red circles) for minimizing the Rosenbrock function via various momentum-based methods. The initial point is indicated by a green circle and the final solution is circled in blue. %\yl{confirm?}
    FR and DY \mq{quickly move towards} the global minimum $(1,1)$ within a few hundred iterations, while PR and HS can not get close to $(1,1)$ even after thousands of iterations.}
    \label{fig:rosen}
\end{figure}

{We examine the momentum-based gradient descent methods.  FRGD is defined in \eqref{eq:CG FR}. We can also replace the FR momentum by PR, HS, DY, as defined in \eqref{eq:CG PR}-\eqref{eq:CG DY}, respectively. 
The step size $\alpha$ plays an important role in the \mq{performance} of these algorithms, as it is tricky to find a proper step size so that the solution does not sway across the parabolic valley $y = x^2$; otherwise it is almost impossible to converge. We carefully tune the step size of each method so that all the algorithms converge in a few hundred steps except for HS and PR which barely converge to the global minimum even after  thousands of iterations, as illustrated in Figure \ref{fig:rosen}. We also compare these four momentum-based methods to gradient descent \eqref{eq:gradient descend}, gradient descent with momentum \eqref{eq: grad momentum}, and Nesterov's gradient \eqref{eq:Nest accel grad}. All these algorithms start from the initial point $(1.2, 1.2)$ and stop when the difference of two consecutive iterates is less than $10^{-8}$ or a maximum number of iterations are achieved.  Table~\ref{tab:rosen} provides the relative errors to the global optimal solution $(1,1)$ and the computational time required by each method, showing that both FR and DY yield the highest accuracy with a reasonable amount of computational time. 
% It is worth noting that for this test we used different step size for each algorithms rather than same step size for all algorithms. If we use a universal step size for all methods, it is almost impossible to achieve any meaningful result, the range of step sized each of these algorithms converges relatively quickly are very different. 
% For general gradient method with constant or Nestrov acceleration, it converges to some local min with step size less than $10^{-2}$ almost all the time give the provided initial point $(1.2, 1.2)$, though how quick it will converge may differ.
% % In practice, we found that choose $10^{-3}$ is overall the best for all of them to converge to global min and relatively quickly. 
% FR momentum converges to the glob if the step size has the order of magnitude around $10^{-5}$, PR and HS is around $10^{-7}$, DY is around $10^{-8}$. 
% FR, PR and HS
% In practice, I find that it is hard to move PR and HS closer to the global min $(1,1)$ with current scheme unless some adaptive step size technique is adopted.
}
\begin{table}[ht]
    \centering
    \begin{tabular}{ccc}
        \hline\hline
        Method & Error&Time \\
        \hline
        GD & 3.5546e-03&2.0594e-01 \\
        GDM & 2.8409e-03&7.8236e-02 \\
        NAG& 1.0201e-03&1.7319e-02 \\
        FRGD & \textbf{7.0804e-04}&2.6686e-02 \\
        PRGD & 1.8049e-02&1.5069e-01 \\
        HSGD & 1.6816e-01&2.8645e-01 \\
        DYGD & 9.0094e-04&\textbf{4.1789e-03} \\
        \hline\hline
    \end{tabular}
    \caption{{Relative errors and computational time of various methods in minimizing the Rosenbrock function. The best results are highlighted in bold.}}
    \label{tab:rosen}
\end{table}

\subsubsection{Quadratic problem}
\begin{figure}[t]
    \centering
    \subfigure[Relative Error]{
    \includegraphics[width = 0.45\textwidth]{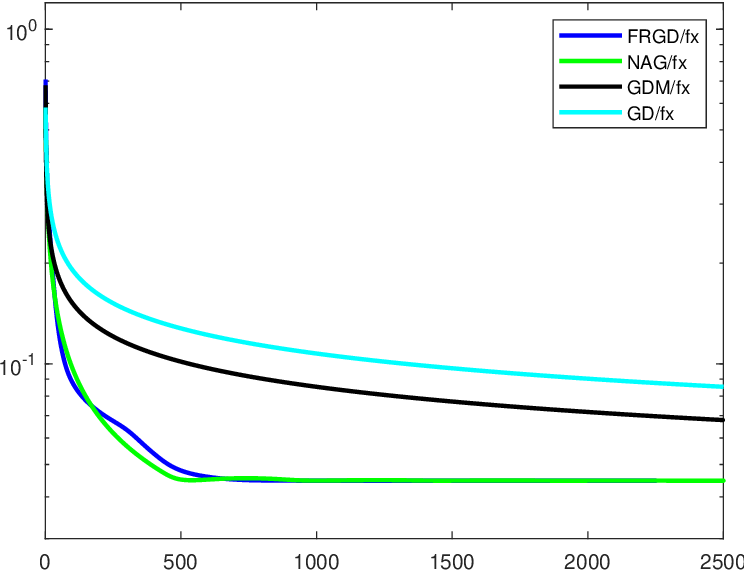}}\qquad
    \subfigure[Objective Function]{
    \includegraphics[width = 0.45\textwidth]{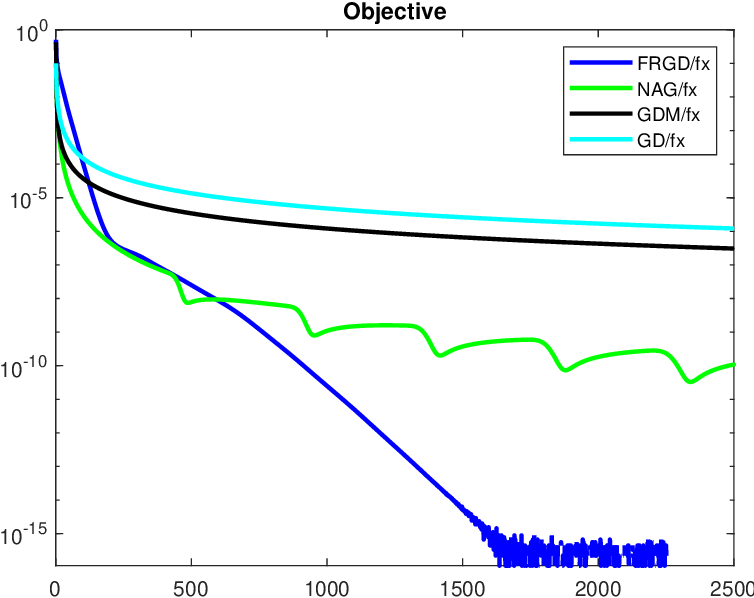}}\\
        \subfigure[Relative Error]{
    \includegraphics[width = 0.45\textwidth]{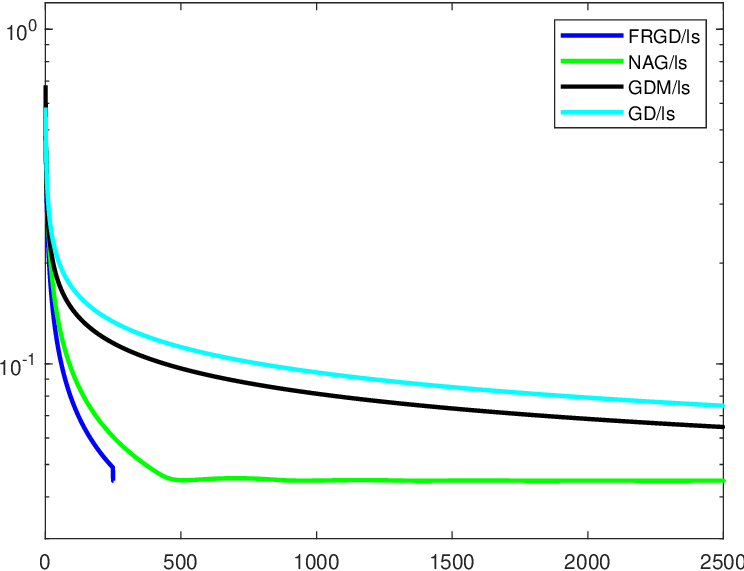}}\qquad
    \subfigure[Objective Function]{
    \includegraphics[width = 0.45\textwidth]{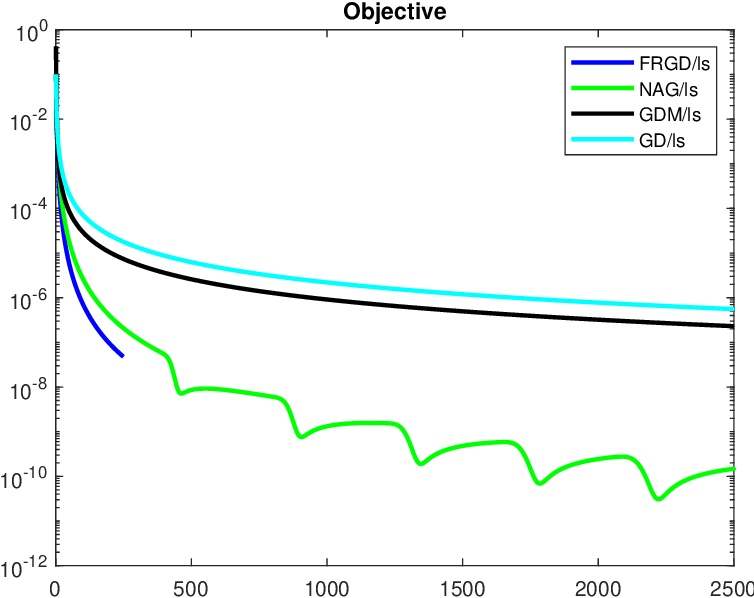}}
    \caption{Comparison of gradient based methods on a quadratic problem with a fixed step size (top) and an adaptive step size by line search (bottom).}
    \label{fig:quadratic fix}
\end{figure}

Following the work of \cite{hardt2014robustness}, 
we consider the quadratic problem \eqref{eq:quadratic} with  $A\in \mathbb{R}^{500\times 500}$ being the Laplacian matrix of a circular graph and $\mathbf{b\in \mathbb{R}^{500}}$ being a vector whose first entry is $1$ and  the remaining entries are $0$. 
It is straightforward to verify that $ g(\h x) $ is convex with Lipschitz constant 4. We
compare GD  \eqref{eq:gradient descend}, GDM \eqref{eq: grad momentum} with a fixed value of $\beta = 0.9$,  NAG  \eqref{eq:Nest accel grad}, and  FRGD \eqref{eq:CG FR}
in terms of relative errors to the ground truth and objective decay. For each competing method, we consider two ways to choose  $\alpha^{(l+1)}$: a fixed value of $0.3$ and an adaptive update via line search \eqref{eq:steepest descend}, indicated by ``fx'' and ``ls'' respectively. For example, FRGD/fx refers to FRGD method with a fixed value, and FRGD/ls refers to FRGD with updating $\alpha^{(l+1)}$ by an exact line search. 
%{\color{red}This is not true. Steepest descent means negative gradient only.}
Note that  FRGD/ls is equivalent to the conjugate gradient for any quadratic problem.
%{\color{red}I do not think conjugate gradient do exact line search.}

We plot the relative errors and the objective functions with respect to iteration numbers of all the competing methods  in  Figure. \ref{fig:quadratic fix}. All the plots are in a logarithmic scale. 
As expected,  GDM yields slightly better performance than SD/GD, while NAG converges significantly faster than GDM, but in an  oscillatory manner.
It is worth noting in  Figure. \ref{fig:quadratic fix} (b)  that the objective values of GD/fx, GDM/fx, NAG/fx (with a fixed step size) become stagnant after  500 iterations, whereas FRGD/fx continues to decay until the machine accuracy.
When the step size is adaptive, FRGD/ls reduces to the classic conjugate gradient that quickly falls into a local minimum, while all the other algorithms require more iterations to converge. In summary, FRGD converges at a rate much faster than regular GD and its variants.

% \subsection{Non-convex Regularizations and Proximal Operators}\label{sect:prox}

%\yl{you can move this paragraph to intro} 
% In this work, we promote the use of non-convex regularizations to find a
% sparse vector when the coherence of $\tensor\Psi$ is relatively large, referred to as  a coherent linear system.
% There are many non-convex alternatives to approximate the $\ell_0$ norm that give
% superior results over the $\ell_1$ norm, such as $\ell_{1/2}$
% \cite{chartrand07,xuCXZ12,laiXY13}, capped $\ell_1$
% \cite{zhang2009multi,shen2012likelihood,louYX16}, transformed
% $\ell_1$~\cite{lv2009unified,zhangX17,zhangX18, GuoLLtransform18}, $\ell_1-\ell_2$  \cite{yinEX14,louYHX14,louY18},  $\ell_1/\ell_2$ \cite{l1dl2,L1dL2_accelerated}, and error function (ERF) \cite{guo2021novel}.
%In \cite{guo2021novel}, Guo et al.~also reported that $\ell_1$-$\ell_2$ is  better than the weighted $\ell_1$ minimization \cite{candes2008enhancing} for solving the $\ell_0$ minimization. 

% We provide the formula of aforementioned regularizations together with their proximal operators as follows,
%For some penalty functions, it has been shown that their proximal operator is unique. }
%\yl{make notations consistent and add l1-l2}

\section{Minimizing the sum of two functions}\label{sect:min-2}
In this section, we focus on minimizing  the sum of two functions defined in \eqref{eq:proximal setup}. Specifically, we consider two different functions of $f$: the $\ell_1$ norm $\|\h x\|_1$ and the $\ell_1-\ell_2$ regularization $\|\h x\|_1-\|\h x\|_2,$ to promote the sparsity of the vector $\h x.$ As $f$  is  not differentiable, we adopt the regular subdifferential for a general (not necessary convex) function \cite[Definition 8.3]{rockafellar2009variational}, {defined by
%\begin{align}
%    \partial  f(\h x) = \{\mathbf{p}| f(\mathbf{z}) \geqslant  f(\h x) + \mathbf{p}\trans  (\mathbf{z-x}) +o(\|\h z-\h x\|)\},
%\end{align}
\begin{align}
    \partial  f(\h x) = \left\{\mathbf{p}|\lim_{\mathbf{z}\rightarrow\h x} {f(\mathbf{z}) -  f(\h x) - \mathbf{p}\trans  (\mathbf{z-x}) \over \|\h z-\h x\|}\geq 0\right\},
\end{align}}
instead of the standard gradient $\nabla f.$   We discretize the gradient flow 
\begin{align}\label{eq:gradient flow}
    \frac{d}{dt} \h x(t) \in -\lambda\partial f({\h x(t)}) - \nabla g(\h x(t)),
\end{align}
that minimizes \eqref{eq:proximal setup} by a semi-implicit scheme as follows,
\begin{align}\label{eq:gradient flow discretized}
    \frac{\h x^{(l+1)} - \h x^{(l)} }{\delta}  \in -\lambda\partial f({\h x^{(l+1)}}) - \nabla g(\h x^{(l)}),
\end{align}
where $\delta>0$ is a step size. The iteration of \eqref{eq:gradient flow discretized} is often referred to as  forward-backward splitting  \cite{combettes2005signal}, as one uses a forward solution in $\nabla g$ and a backward one in $\partial f$. After
rearranging \eqref{eq:gradient flow discretized}, we obtain
\[
\h x^{(l+1)} \in \big(I + \delta\lambda\partial f\big)^{-1} \big( \h x^{(l)} - \nabla g(\h x^{(l)})\big),
\]
which implies that $\h x^{(l+1)}$ is an optimal solution to
\begin{equation}
 \h x^{(l+1)} \in \argmin_{\h x} \delta\lambda f(\h x) + \frac 1 2 \|\h x-\h x^{(l)}+\delta \nabla g(\h x^{(l)})\|_2^2.   \label{eq:before_prox}
\end{equation}
The solution to \eqref{eq:before_prox} can be expressed by the corresponding proximal operator. Recall that a proximal operator \cite{parikh2014proximal}   of a functional $J(\cdot)$ with a positive parameter $\mu>0$ is defined by 
\begin{equation}\label{eq:prox-def}
\mathbf{prox}_{J} (\h x; \mu) = \argmin_{\h y} \Big(\mu J(\h y)+\frac 1 2 \|\h x-\h y\|_{2}^{2}\Big).
\end{equation}
Now by relating Equation. \eqref{eq:prox-def} \eqref{eq:before_prox}, we have an iterative update,
\begin{align}
    \h x^{(l+1)} \in \mathbf{prox}_{f}\left( \h x^{(l)} -\delta \nabla g(\h x^{(l)});\delta\lambda\right).
    \label{eq:prox def}
\end{align}
For $f(\h x) = \|\h x\|_{1}$, its proximal operator is given by
\begin{align}
 \label{eq:l_1 prox} \mathbf{prox}_{\ell_1}(\h x;\mu) = \sign(\h x)\circ \max(|\h x|-\mu,0), 
\end{align}
where $\circ$ denotes the Hadamard operator for componentwise operation.
As the proximal operator for $\ell_1$
 is called soft shrinkage, the corresponding iteration \eqref{eq:prox def}  is referred to as iterative soft-thresholding algorithm (ISTA) \cite{chambolle1998nonlinear,figueiredo2003algorithm,daubechies2004iterative,hale2007fixed,vonesch2007fast,wright2009sparse}. One accelerated scheme of ISTA is called fast iterative soft-thresholding algorithm (FISTA) \cite{beck2009fast}. It is a momentum based algorithm that utilized  the Nesterov's update \eqref{eq:Nest accel grad} on the step size, having the  form, 
\begin{align}
    \begin{cases}
    t^{(l+1)} &= \frac{1 +\sqrt{4(t^{(l)})^{2} + 1}}{2}\\
    \h y^{(l+1)} &= \h x^{(l)} + \frac{t^{(l)}-1}{t^{(l+1)}} (\h x^{(l)} - \h x^{(l-1)})\\
    \h x^{(l+1)} &= \mathbf{prox}_{\ell_1} \left( \mathbf{y}^{(l+1)} -\delta \nabla g(\mathbf{y}^{(l+1)}); \delta\lambda\right).
    \end{cases}\label{eq:update FISTA}
\end{align}
The momentum term in FISTA is proven to be efficient, but the algorithm exhibits oscillatory patterns during the minimization process. To have a guaranteed descent, the accelerated proximal gradient (APG)  algorithm \cite{li2015accelerated} compares the objective function at two proximal solutions and selects the smaller one. In short, the APG algorithm
goes as follows,
\begin{align}
    \begin{cases}
    t^{(l+1)} &= \frac{1+\sqrt{4(t^{(l)})^{2} + 1}}{2}\\
    \h y^{(l+1)} &= \h x^{(l)} +\frac{t^{(l)}}{t^{(l+1)}} (\h u^{(l)} - \h x^{(l)})+\frac{t^{(l)}-1}{t^{(l+1)}} (\h x^{(l)} - \h  x^{(l-1)})\\
    \h u^{(l+1)} &= \mathbf{prox}_{f} \left( \mathbf{y}^{(l+1)} -\delta \nabla g(\mathbf{y}^{(l+1)}); \delta\lambda\right)\\
    \h v^{(l+1)} &= \mathbf{prox}_{f} \left( \mathbf{x}^{(l)} -\delta \nabla g(\mathbf{x}^{(l)}); \delta\lambda\right)\\
    \h x^{(l+1)} &= \argmin\limits_{\h z \in \{\h u^{(l+1)}, \h v^{(l+1)}\}} \lambda f(\h z) + g(\h z).
    \end{cases}\label{eq:APG update}
\end{align}
% APG has a provable faster convergence rate than FISTA \cite{nocedal2006numerical}, {\color{red} This citation is not correct. Also, I do not think APG is proven to be faster than FISTA.}  but empirically we observe that it does not necessarily work better than FISTA.

We propose to combine adaptive momentum formula and FISTA for minimizing the general problem of \eqref{eq:proximal setup}.  In particular, we replace the FISTA momentum update \eqref{eq:update FISTA} in terms of FR, thus leading to
\begin{align}
    \begin{cases}
        \beta^{(l+1)} &= \frac{\|\nabla g(\h x^{(l)})\|^{2}}{\|\nabla g(\h x^{(l-1)})\|^{2}}, \\
        \mathbf{y}^{(l+1)} &= \h x^{(l)} +  \beta^{(l+1)}(\h x^{(l)} - \h x^{(l-1)}),\\
        \h x^{(l+1)} &\in \mathbf{prox}_{f} \left( \mathbf{y}^{(l+1)} -\delta \nabla g(\mathbf{y}^{(l+1)}); \delta\lambda\right).
    \end{cases}\label{eq:update FR}
\end{align}
 Similarly we can use other momentum terms given in \eqref{eq:CG PR}-\eqref{eq:CG DY}. The proximal operator for the $\ell_1$ norm is given in \eqref{eq:l_1 prox}, while the proximal operator for $\ell_{1}-\ell_{2}$  \cite{louY18} can be defined separately  into the following cases,
\begin{itemize}
    \item If $\|\bm y \|_{\infty} > \lambda$, one has $\mathbf{prox}_{\ell_{1-2}}(\bm y;\lambda) = \frac{\bm z (\|\bm z\|_{2} + \lambda)}{ \|\bm z\|_{2} }$, where $\bm z = \mathbf{prox}_{\ell_1}(\bm y;\lambda)$.
    \item If $\|\bm y \|_{\infty} \leq \lambda$, then $\bm c^{*}:=\mathbf{prox}_{\ell_{1-2}}(\bm y;\lambda)$ is an optimal solution if and only if   $ c^{*}_{i} = 0$ for $|y_{i}| < \|\bm y \|_{\infty}$, $\|\bm c^{*}\|_{2} = \|\bm y \|_{\infty},$ and $ c^{*}_{i}y_{i} \geq 0$ for all $i$. The optimality condition implies infinitely many solutions of $\bm c^*,$ among which we choose $c^{*}_{i} = \sign(y_{i}) \|\bm y\|_{\infty}$ for the smallest $i$ satisfies $|y_{i}| = \|\bm y \|_{\infty}$ and the rest coefficients set to be zero.
   % \item  If $\|\bm y \|_{\infty} < \lambda$, $\bm c^{*}$ is an optimal solution if and only if $ c^{*}_{i} = 0$ for all $|y_{i}| <  \|\bm y \|_{\infty} $, $\|\bm c^{*}\|_{2} = \|\bm y \|_{\infty} $ and $ c^{*}_{i}y_{i} \geq 0$ for all $i$. \yl{how we choose in this case?}
  %  \item  When $\|\bm y \|_{\infty} < (1-\alpha)\lambda $, $\bm c^{*} = 0$
\end{itemize}

%  We provide the formula of aforementioned regularizations together with their proximal operators as follows,
% %For some penalty functions, it has been shown that their proximal operator is unique. }
% %\yl{make notations consistent and add l1-l2}
% \begin{itemize}
% \item The $\ell_{1}$ norm of $\h y$ is   $\|\bm y\|_1$ with its given by
% \begin{align}
%  \label{eq:l_1 prox} \mathbf{prox}_{\ell_1}(\bm y;\lambda) = \sign(\bm y)\circ \max(|\bm y|-\lambda,0), 
% \end{align}
% where $\circ$ denotes the Hadamard operator for componentwise operation.
% \item The  regularization is defined by $\|\bm y \|_{1} - \|\bm y \|_{2},$ and % for $\theta \in [0,1]$ is a parameter that we choose $\alpha = 1$ for all our tests.
% its proximal operator is 
% \end{itemize}

%\section{Experiments}\label{sect:acc-exp}
In what follows, we present  experimental results on  the convex $\ell_1$ minimization in Section. \ref{sect:exp_l1} and the non-convex $\ell_1-\ell_2$ minimization in Section. \ref{sect:exp_l12}, respectively.

\subsection{Convex $\ell_{1}$ Minimization}
\label{sect:exp_l1}

We test the performance of various methods to minimize the $\ell_1$ norm with the least-squares fitting term, i.e., %\yl{change gamma to lambda for constructed solution}
\begin{equation}\label{eq:l1-uncon}
\h x^*=\argmin_{\h x}\lambda\|\h x\|_1 + \frac 1 2 \|A\h x-\h b\|_2^2.
\end{equation} 
We generate the sensing matrix $A$ from Gaussian random matrices and  a ground-truth sparse vector $\h x$  of sparsity $5$. Compressive sensing often involves an under-determined linear system, which implies that the matrix $A$ has more columns than rows (a fat matrix). Here we examine both under-determined (fat) and over-determined (tall) matrices with size $256\times 1024$ and $1024\times 256,$ respectively.  
% We compare our proposed algorithm  \eqref{eq:update FR} with different momentum updates (FR, PR, HS, and DY) against FISTA   \eqref{eq:update FISTA} and APG  \eqref{eq:APG update}. \yl{We fix $\delta$ to be XXX for all the algorithms.} 

We consider two ways to generate the data vector $\h b.$ One is a standard \textit{sparse recovery} setting, in which $\h b$ is obtained by matrix-vector multiplication ($A\h x$) with additive Gaussian noise of $30$ dB.
% Another is referred to as a \textit{constructed} case in a way that a given vector $\h x$ corresponds to a stationary point of the problem \eqref{eq:l1-uncon} with a constructed data vector $\h b$. 
{Another is referred to as a \textit{constructed} case following the work of ~\cite{lorenz2011constructing}. In particular, we construct a data vector $\h b$ such that a specific sparse vector $\h x^*$ is a stationary point of \eqref{eq:l1-uncon} when given a positive parameter $\lambda$ and a matrix $A$.}
Any non-zero stationary point satisfies the following first-order optimality condition: 
\begin{equation}\label{eq:uncon_construct}
\textstyle \lambda \mathbf{p}^*+A\trans(A\h x^*-\h b)=\h 0,
\end{equation}
where $\mathbf{p}^*\in\partial\|\h x^*\|_1$. Denote Sign$(\cdot)$ as the multi-valued sign, \emph{i.e.},
\begin{equation}
\mathbf{y}\in\mbox{Sign}(\h x)\ \Longleftrightarrow\ y_i\left\{\begin{array}{ll}
=1, & \mbox{if} \  x_i>0,\\
=-1, & \mbox{if} \  x_i<0,\\
\in[-1,1], & \mbox{if} \ x_i=0.
\end{array}\right.
\end{equation}
Given $A,~\lambda$, and  $\h x^*$, we want to find $\h x\in \mbox{Sign}(\h x^*)$ and $\h x \in \mbox{Range}(A\trans)$. 
If $\mathbf{y}$ satisfies $A\trans \mathbf{y}=\h x $ and $\mathbf{b}$ is defined by $\mathbf{b}=\lambda \mathbf{y}+A\h x^*$, then $\h x^*$ is a stationary point to~\eqref{eq:l1-uncon}. 
To find $\h x\in\mathbb R^N$,
we adopt the iteration
% we consider the projection onto convex sets (POCS)~\cite{cheney1959proximity} by alternatively projecting onto two convex sets: $\h x\in \mbox{Sign}(\h x^*)$ and $\h x \in \mbox{Range}(A\trans)$. In particular, we compute the orthogonal basis of $A\trans$, denoted as $U$, for the sake of projecting onto the set $\mbox{Range}(A\trans)$. The iteration starts with $\h x^{(0)}\in\mbox{Sign}(\h x^*)$ and proceeds  
\begin{align}\label{eq:constructedl1}
%	&&\textstyle v^{k} = UU\trans(w^k - \frac {x^*}{\|x^*\|_2}) + {x^*\over \|x^*\|_2}\\
	 \h x^{(k+1)} = P_{\mbox{Sign}(\h x^*)}\left(UU\trans\left(\h x^{(k)} \right) \right), %+ {\frac{\h x^*} {\|\h x^*\|_2}}
\end{align}
until  a stopping criterion is reached. Please refer to \cite{lorenz2011constructing} for more details.
%The stopping condition for POCS is either $\|\h x^{(k+1)}-\h x^{(k)}\|_2<1e^{-10}$ or $k>10N$. % {\color{red}We do not need to explain the details here. Just need to mention the created $\h x$ is an optimal solution to the optimization problem.}

%Constructed solution in  Figure. \ref{fig:l1con} and sparse recovery problem in  Figure. \ref{fig:l1rand}. 

\begin{figure}[t]
  \centering
  \subfigure[Relative Error]{
    \includegraphics[width = 0.45\textwidth]{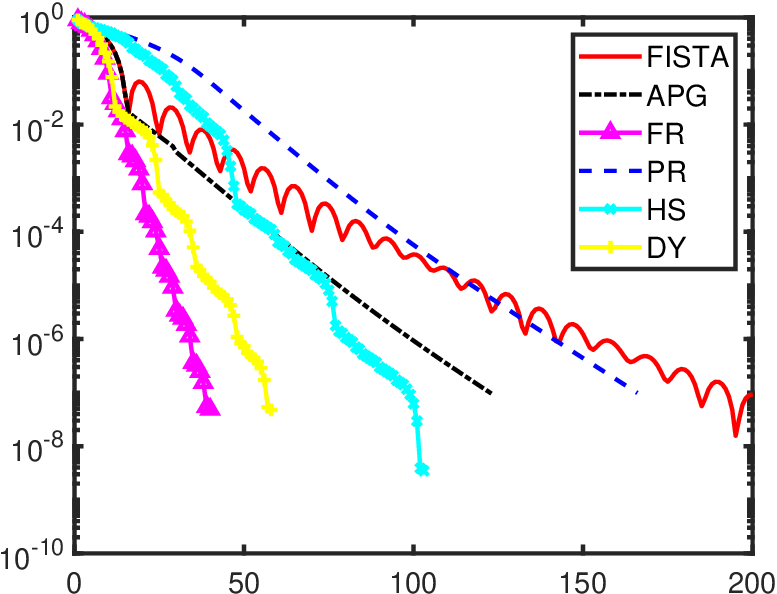}}\qquad
    \subfigure[Objective Function]{
    \includegraphics[width = 0.45\textwidth]{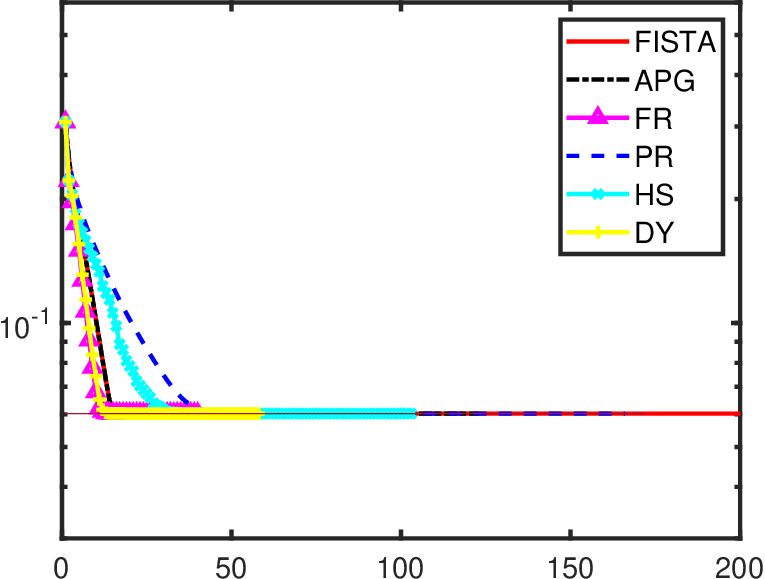}}\\
      \subfigure[Relative Error]{
    \includegraphics[width = 0.45\textwidth]{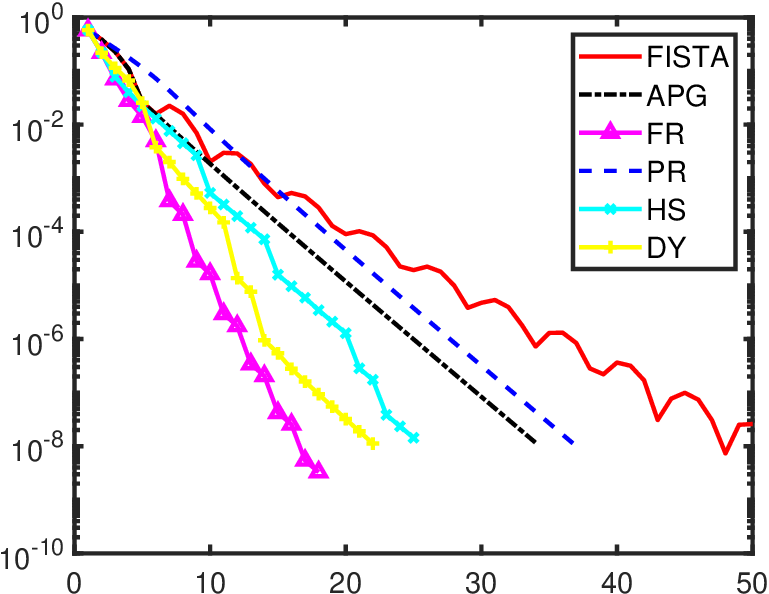}}\qquad
    \subfigure[Objective Function]{
    \includegraphics[width = 0.45\textwidth]{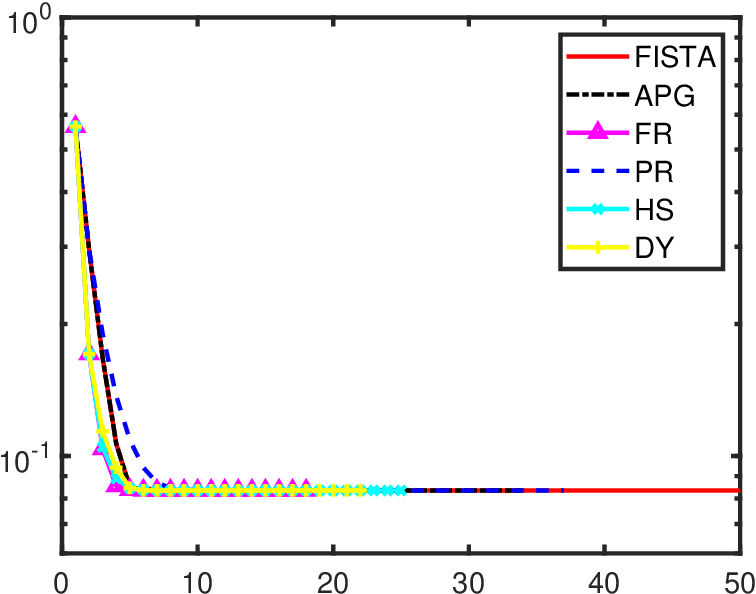}}
    \caption{Comparison of  $\ell_{1}$ minimization methods using a $256 \times 1024$ (top) and $1024\times 256$ (bottom) matrix $A$ in a constructed case. }
    \label{fig:l1con fat}
\end{figure}

The constructed setting is examined in
 Figure. \ref{fig:l1con fat} that contains both  fat and tall matrices. We use a fixed value of $\lambda$ for all the algorithms so that they solve the same problem and we 
 tune $\delta$ to achieve the fastest convergence. Specifically we choose the best $\delta$ among the set $\{10^{-4}$, $10^{-3}$, $\cdots$, $10^{1}\}$ that achieves the smallest objective function value when convergent.
The proposed   method with the FR momentum converges the fastest among all the other methods. 
FISTA initially converges faster than  APG and PR/HS, while it always oscillates no matter whether the matrix $A$ is fat or tall.

% {
% First, we begin by setting each component of $\lambda$ to $10^{-4}$, $10^{-3}$, $\cdots$, $10^{1}$ and a small step size $\delta = 10^{-3}$ to determine on which order of magnitude $lambda$ will achieve the smallest objective function value. This way we find an scalar $r\in\{-4,-3,\cdots,1\}$. After that, we use $10^{r}$ multiplied by $0.5,1.0,1.5,\cdots,9.5$ to find the best result and repeat again to achieve an accuracy of $0.1$. After we find the best $\lambda$, we fix it for all testing cases and repeat the above process to determine the best $\delta$ for each different algorithm.}

\begin{figure}[t]
  \centering
  \subfigure[Relative Error]{
    \includegraphics[width = 0.45\textwidth]{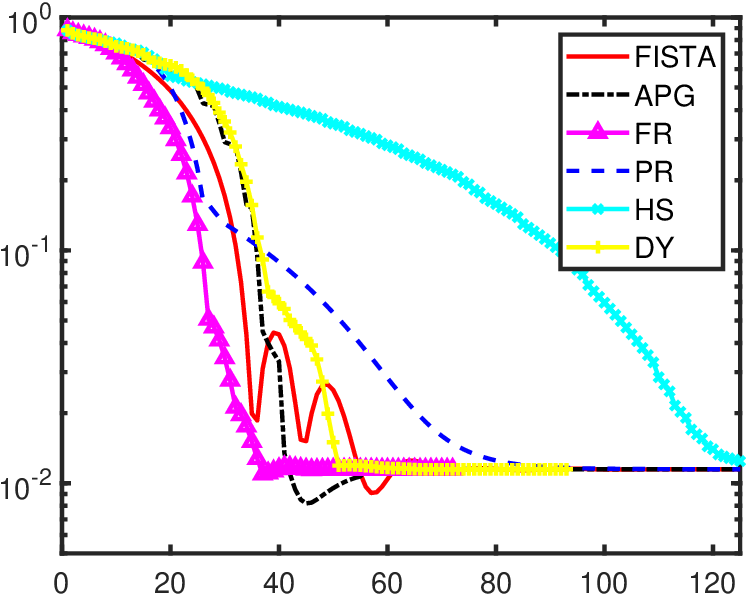}}\qquad
    \subfigure[Objective Function]{
f    \includegraphics[width = 0.45\textwidth]{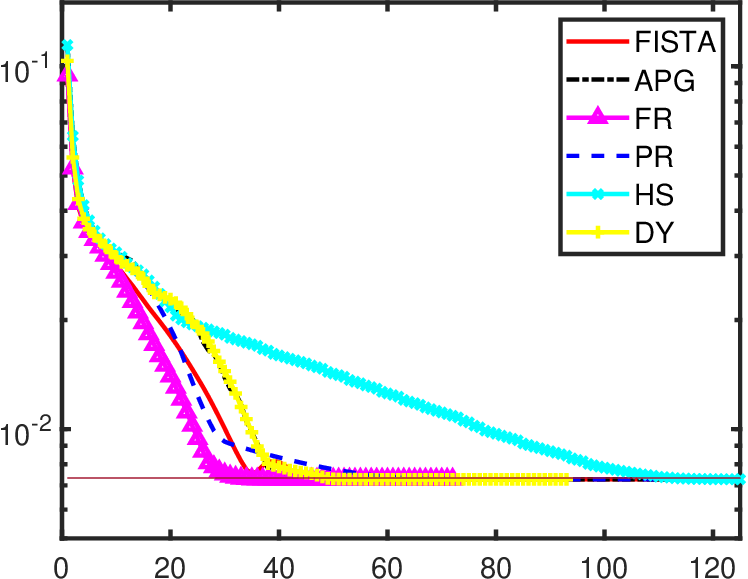}}\\
      \subfigure[Relative Error]{
    \includegraphics[width = 0.45\textwidth]{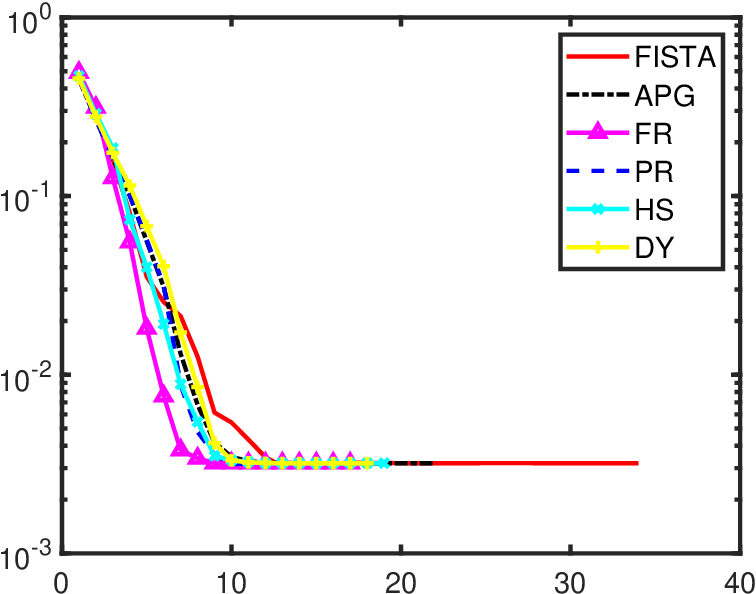}}\qquad
    \subfigure[Objective Function]{
    \includegraphics[width = 0.45\textwidth]{L1WithTallMatix20220630/L1_error_1024by256_30dB_rand.eps}}
    \caption{Comparison of  $\ell_{1}$ minimization methods using a $256 \times 1024$ (top) and $1024\times 256$ (bottom) matrix $A$ in a standard sparse recovery setting.}
    \label{fig:l1rand fat}
\end{figure}

We examine a standard sparse recovery setting where the data $\h b$ is obtained by matrix-vector multiplication with additive noise. Again two types of matrices are considered, a  $256\times 1024$ (fat) matrix and a $1024\times 256$ (tall)  one. We use the proposed algorithm with a small step size $\delta = 10^{-3}$ to find the optimal $\lambda$ value among the set $\{10^{-4}$, $10^{-3}$, $\cdots$, $10^{1}\}$  that yields the smallest objective function value. Then we fix this optimal $\lambda$ for all the competing algorithms while tuning the $\delta$ parameter in the same way as in the constructed case. The results are {presented} in  Figure. \ref{fig:l1rand fat}.
We   observe that our proposed algorithm has still some advantages over FISTA and APG. %{\color{red}what does edge mean?}
Interestingly, in  Figure. \ref{fig:l1rand fat} (a), we observe that the APG algorithm  performs worse than FISTA until about 40th iteration due to the oscillatory nature of FISTA.
This phenomenon implies that APG is less robust than FISTA in certain applications, which is somewhat counter-intuitive. In this set of experiments, all the methods can not reach the accuracy of $10^{-3}$ in terms of relative errors, as compared to the constructed cases ($10^{-8}$). This is because the ground-truth solution may not be a stationary point to the corresponding minimization problem. 

{The CPU time required by various methods is reported in Table~\ref{tab:CPU time l1}, which includes different shapes of the sensing matrix $A$ (a $256\times 1024$ fat matrix or a $1024\times 256$ tall matrix) as well two testing scenarios (a constructed case and a standard setting in spare recovery). The computational time of a larger dimension of $1024\times 4096$ and $4096\times 1024$ is recorded in Table~\ref{tab:CPU time l1-large}. 
Both FR and DY are the winners in terms of computational efficiency, which is consistent with our observation in the Rosebrock function. }

 \begin{table}[]
     \centering
     \begin{tabular}{ccccc}
     \hline\hline
          Method & Constructed fat& Standard fat &constructed tall  &Standard tall \\
     \hline
          FISTA & 1.5727e-01& 2.1449e-01 &2.6440e-02 & 1.7032e-02\\
          APG & 1.8135e-01 & 3.7542e-01 & 3.5753e-02 & 2.1801e-02\\
          FR & 6.7575e-02& 1.2839e-01 &1.9097e-02& 1.7562e-02\\
          PR & 2.6586e-01& 3.1863e-01 &3.8764e-02& 1.9722e-02\\
          HS & 1.6053e-01& 3.0457e-01 &2.5498e-02& 1.8789e-02\\
          DY & 9.2728e-02& 1.5137e-01 &2.1915e-02& 1.7746e-02\\
          \hline
     \end{tabular}
     \caption{{CPU time for various $\ell_{1}$ minimization methods 
     using a $256 \times 1024$ (fat) matrix $A$ or a $1024\times 256$ (tall) matrix under either a constructed or standard setting.} }
     \label{tab:CPU time l1}
 \end{table}

 \begin{table}[]
     \centering
     \begin{tabular}{ccccc}
     \hline\hline
          Method & Constructed fat& Standard fat &constructed tall  &Standard tall \\
     \hline
          FISTA& 4.6661e+00 & 4.6932e+00 & 8.7736e-01 & 6.7252e-01\\
          APG& 4.8702e+00 & 4.9007e+00  & 1.0459e+00 & 6.9235e-01\\
          FR& 1.6858e+00 & 2.4300e+00 &5.5734e-01& 6.4436e-01\\
          PR& 6.9300e+00 & 5.7933e+00 &1.1392e+00& 6.4442e-01\\
          HS& 3.7369e+00 & 5.9089e+00 &7.2674e-01& 7.0761e-01\\
          DY& 2.2746e+00 & 2.8957e+00 &6.1895e-01& 5.7265e-01\\
          \hline
     \end{tabular}
     \caption{{CPU time for various $\ell_{1}$ minimization methods 
     using a $1024 \times 4096$ (fat) matrix $A$ or a $4096\times 1024$ (tall) matrix under either a constructed or standard setting.} }
     \label{tab:CPU time l1-large}
 \end{table}

\subsection{Non-convex $\ell_{1}-\ell_{2}$ Minimization}
\label{sect:exp_l12}

\begin{figure}[t]
  \centering
    \subfigure[Relative Error]{
    \includegraphics[width = 0.4\textwidth]{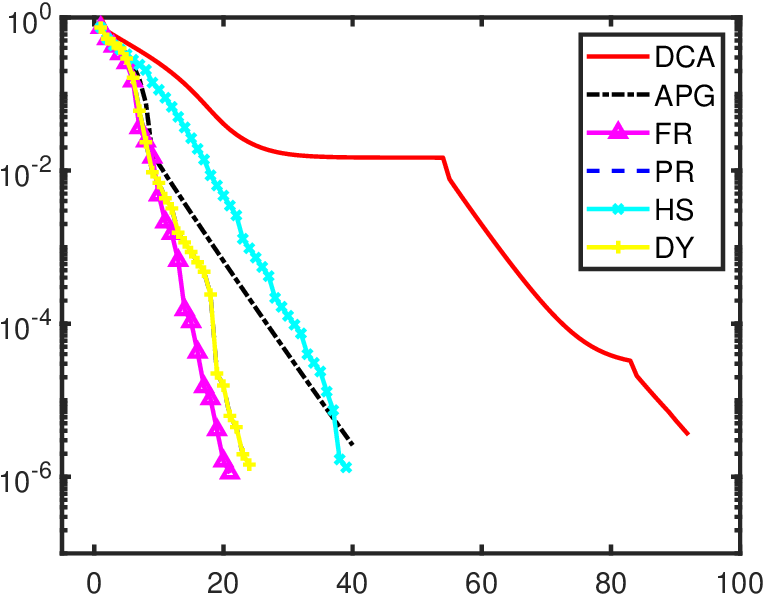}}
    \subfigure[Objective Function]{
    \includegraphics[width = 0.4\textwidth]{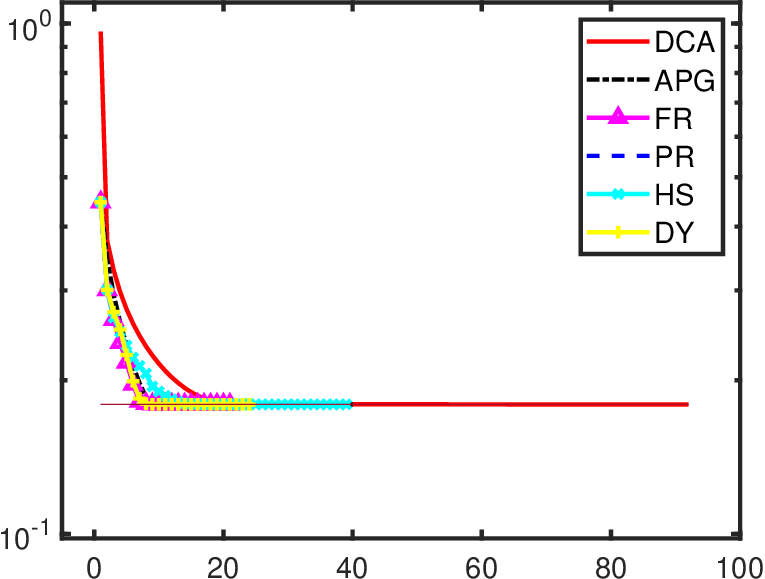}}\\
    \subfigure[Relative Error]{
    \includegraphics[width = 0.4\textwidth]{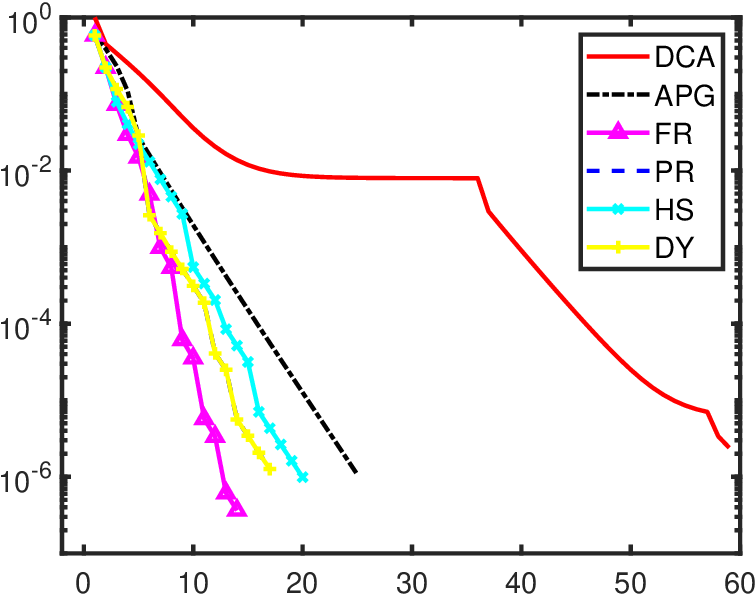}}
    \subfigure[Objective Function]{
    \includegraphics[width = 0.4\textwidth]{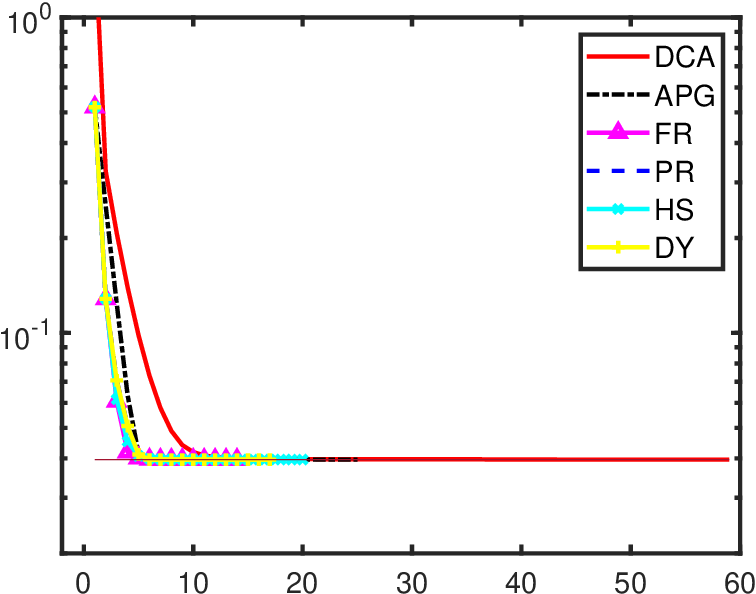}}
    \caption{Comparison of  $\ell_{1}-\ell_2$ minimization methods using a $256 \times 1024$ (top) and $1024\times 256$ (bottom) matrix $A$ in a constructed case.}
    \label{fig:l12con tall}
\end{figure}

Lastly, we consider the non-convex $\ell_{1}-\ell_{2}$ minimization problem,
\begin{align} \label{eq: obj dca}
F(\h x) = \lambda (\Vert \h x \Vert_1- \Vert \h x \Vert_2) + \frac 1 2\Vert\tensor A\h x-\h b\Vert_2^2,
\end{align}
which was originally solved by the difference of convex algorithm (DCA) \cite{TA98,phamLe2005dc}. As a baseline algorithm for comparison, we give a brief description of DCA. After decomposing $F$ into a difference of two convex functions, DCA further relies on 
% The idea of DCA is to
% decompose $F$ in \eqref{eq: obj dca} as a difference of two convex functions
% $F(\h x) = R(\h x) - H(\h x)$ where
% % $F(\h x) = G(\h x) - H(\h x)$ where
% 	\begin{equation*}\label{dcdecomp}
% 	\left\{\begin{array}{l}
% 	R(\h x) =\frac{1}{2}\|\tensor A\h x - \h u\|_2^2 + \lambda\|\h x\|_1 \\
% 	H(\h x) = \lambda\|\h x\|_2.
% 	\end{array}\right.\end{equation*}
% An iterative scheme of minimizing $F(\h x)$ relies on 
the linearization at the current step $\h x^{(l)}$ to advance to the next one, i.e., 
	\begin{equation}\label{eq:iter}
	\h x^{(l+1)} =\argmin_{\h x\in\mathbb R^{n}} \frac{1}{2}\|\tensor A\h x -\h b\|_2^2   + \lambda\| \h x\|_1-\left\langle \h x,\frac{\lambda \h x^{(l)}}{\|\h x^{(l)}\|_2}\right\rangle.
	\end{equation}
% Note that \eqref{eq:iter} is an $\ell_1$ type of problems, which can be minimized by ADMM.

% The $\ell_{1}-\ell_{2}$ regularization is defined by $\|\bm y \|_{1} - \|\bm y \|_{2},$ and % for $\theta \in [0,1]$ is a parameter that we choose $\alpha = 1$ for all our tests.
% its proximal operator \cite{louY18} is separated into the following cases,
% \begin{itemize}
%     \item If $\|\bm y \|_{\infty} > \lambda$, one has $\mathbf{prox}_{\ell_{1-2}}(\bm y;\lambda) = \frac{\bm z (\|\bm z\|_{2} + \lambda)}{ \|\tensor z\|_{2} }$, where $\bm z = \mathbf{prox}_{\ell_1}(\bm y;\lambda)$.
%     \item If $\|\bm y \|_{\infty} \leq \lambda$, then $\bm c^{*}:=\mathbf{prox}_{\ell_{1-2}}(\bm y;\lambda)$ is an optimal solution if and only if   $ c^{*}_{i} = 0$ for $|y_{i}| < \|\bm y \|_{\infty}$, $\|\bm c^{*}\|_{2} = \|\bm y \|_{\infty},$ and $ c^{*}_{i}y_{i} \geq 0$ for all $i$. The optimality condition implies infinitely many solutions of $\bm c^*,$ among which we choose $c^{*}_{i} = \sign(y_{i}) \|\bm y\|_{\infty}$ for the smallest $i$ satisfies $|y_{i}| = \|\bm y \|_{\infty}$ and the rest coefficients set to be zero.
%   % \item  If $\|\bm y \|_{\infty} < \lambda$, $\bm c^{*}$ is an optimal solution if and only if $ c^{*}_{i} = 0$ for all $|y_{i}| <  \|\bm y \|_{\infty} $, $\|\bm c^{*}\|_{2} = \|\bm y \|_{\infty} $ and $ c^{*}_{i}y_{i} \geq 0$ for all $i$. \yl{how we choose in this case?}
%   %  \item  When $\|\bm y \|_{\infty} < (1-\alpha)\lambda $, $\bm c^{*} = 0$
% \end{itemize}

To generate a constructed solution for the $\ell_{1}-\ell_{2}$ problem, we only need to replace the iteration \eqref{eq:constructedl1} for constructing the $\ell_1$ solution by
\begin{align}\label{eq:constructed-l12}
	 \h x^{(k+1)} = P_{\mbox{Sign}(\h x^*)}\left(UU\trans\left(\h x^{(k)} - \frac{\h x^*} {\|\h x^*\|_2} \right) + {\frac{\h x^*} {\|\h x^*\|_2}} \right).
\end{align}
Due to the non-convex nature of $F(\h x)$, the iteration \eqref{eq:constructed-l12} may not converge and $\h x^*$ may not exist. 
The results of $\ell_1-\ell_2$ minimization methods on a constructed case are illustrated in  Figure. \ref{fig:l12con tall} for matrix  sizes of $256\times 1024$ and $1024\times 256.$ %respectively. 
% Figure. \ref{fig:l12con tall}(a) (b) shows the relative error and objective functions of different algorithms with a fat matrix. 
Note that DCA is a doule-loop algorithm and its iteration number is counted as inner loop iterations, and yet
 the original DCA implementation \cite{yinLHX14,louYX16} is the slowest, followed by    APG. % and FBS Restart. 
Our proposed algorithm is the fastest, having a clear advantage over all the  other algorithms.
% Figure. \ref{fig:l12con tall}.(a) (b) shows the results with a tall matrix and our proposed algorithm holds the same advantage.

 Figure. \ref{fig:l12rand tall} shows the results for a sparse recovery problem.  DCA is still the slowest, while our method is the fastest. 
%  Figure. \ref{fig:l12rand tall}.(a) (b) shows the relative error and objective functions of different algorithms with a fat matrix. We can see in that this case the slowest algorithm is still DCA. 
 The proposed method is worse than  DCA for a tall matrix. This may  attribute to the fact that the ground-truth signal is not the optimal solution to \eqref{eq: obj dca}, and as a result, the performance is rather random.  
 {The CPU time for these $\ell_1-\ell_2$ minimization methods is listed in Table~\ref{tab:CPU time l1-l2}. We observe that all the momentum-based methods seem comparable in computational time, while DCA is the slowest.  }
 %{\color{red}It is unfair to compare the iteration number for DCA and other algorithms, because DAC takes much longer time than other algorithms for one iteration.}

\begin{figure}[t]
  \centering
  \subfigure[Relative Error]{
    \includegraphics[width = 0.4\textwidth]{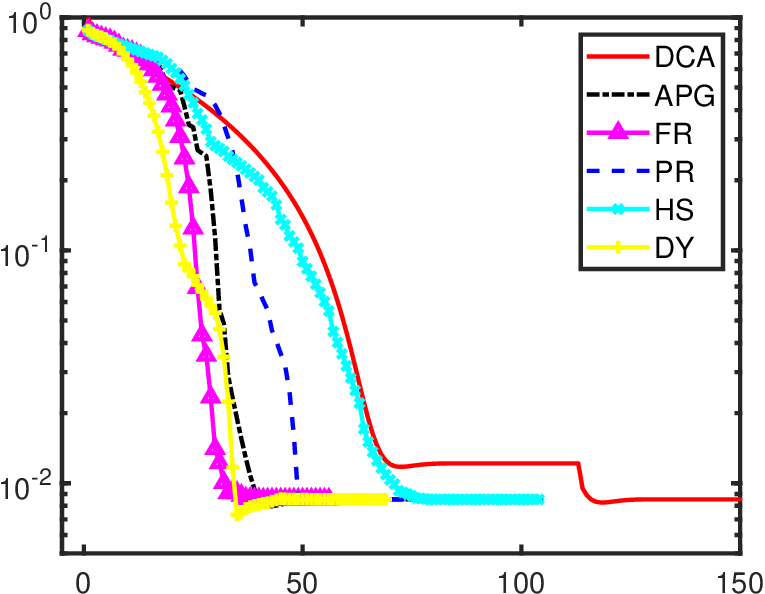}}
    \subfigure[Objective Function]{
    \includegraphics[width = 0.4\textwidth]{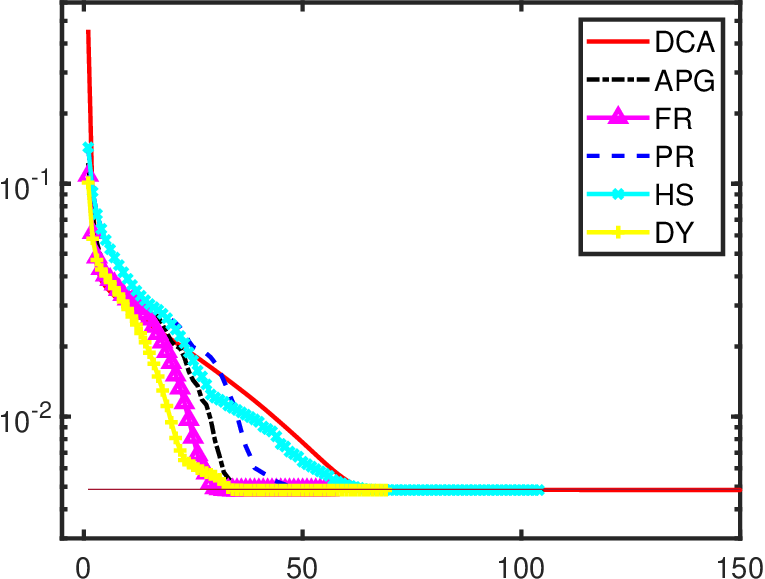}}\\
    \subfigure[Relative Error]{
    \includegraphics[width = 0.4\textwidth]{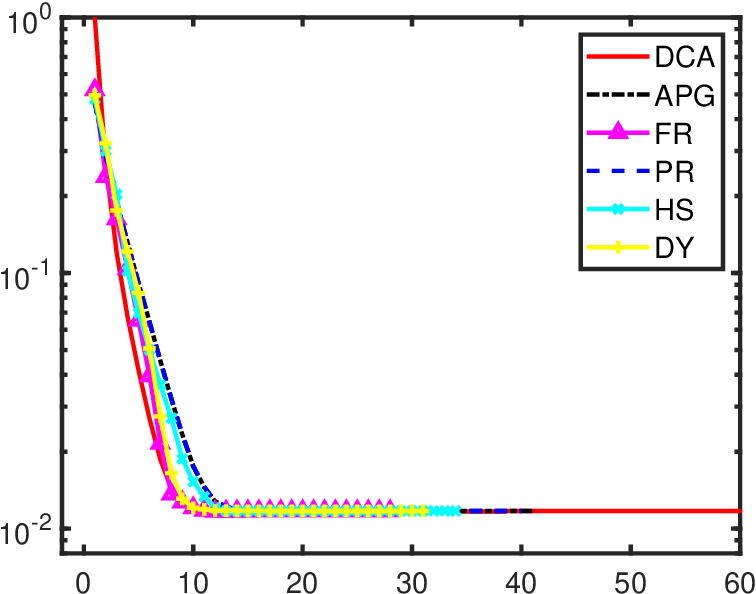}}
    \subfigure[Objective Function]{
    \includegraphics[width = 0.4\textwidth]{L12WithTallMatix20220420/L12_error_1024by256_30dB_rand.eps}}
    \caption{Comparison of  $\ell_{1}-\ell_2$ minimization methods using a $256 \times 1024$ (top) and $1024\times 256$ (bottom) matrix $A$ in a standard sparse recovery setting.}
    \label{fig:l12rand tall}
\end{figure}

 \begin{table}[]
     \centering
     \begin{tabular}{ccccc}
     \hline\hline
          Method & Constructed fat &random fat & Constructed tall &random tall \\
     \hline
          DCA & 1.2112e-02& 5.4499e-02 &7.3475e-02 &8.5999e-02\\
          APG & 9.5120e-03 & 2.1489e-01 & 3.3662e-02 & 2.1801e-02\\
          FR & 6.0010e-03& 1.2058e-01 &2.0455e-02& 6.2328e-02\\
          PR & 7.3012e-03& 1.7407e-01 &2.3123e-02& 6.1710e-02\\
          HS & 1.0967e-02& 2.4206e-01 &2.9486e-02& 5.4402e-02\\
          DY & 5.9523e-03& 1.6031e-01 &2.4117e-02& 4.4513e-02\\
          \hline
 \end{tabular}
     \caption{{CPU time for various $\ell_{1}-\ell_2$ minimization methods 
     using a $256 \times 1024$ (fat) matrix $A$ or a $1024\times 256$ (tall) matrix under either a constructed or standard setting.}}
     \label{tab:CPU time l1-l2}
 \end{table}

\section{Conclusion}\label{sect:conclude}
In this paper, we leveraged adaptive momentum from nonlinear conjugate gradient algorithms for the purpose of acceleration. Unlike the existing works that rely on line search to establish convergence of gradient-based algorithms, we proposed the use of a fix step size  and proved the convergence of FRGD on a quadratic problem. In addition, we combined the adaptive momentum with FISTA to deal with non-smooth objective function. 
The resulting algorithm has a relatively simple FISTA-like structure. We demonstrated the accelerated phenomena of the proposed approach over FISTA and APG on a convex $\ell_1$ minimization and a nonconvex $\ell_1-\ell_2$ problem for sparse recovery.

\begin{acknowledgements} 
YL acknowledges the support from NSF CAREER DMS-1846690. BW is supported by NSF DMS-1924935, DMS-1952339, DMS-2152762, and DMS-2208361. BW also acknowledges support from the office of Science of the department of energy under grant number DE-SC0021142 and DE-SC0002722. MY was supported by the NSF DMS-2012439 and Shenzhen Science and Technology Program  ZDSYS20211021111415025. XY acknowledges the support from NSF CAREER DMS-2143915. XY was also partly supported by DOE, Office of Science, Office of Advanced
Scientific Computing Research (ASCR) as part of Multifaceted Mathematics for Rare, Extreme Events in Complex Energy and Environment Systems (MACSER). QY acknowledges the support from the NSF grant DMS-1821144.
\end{acknowledgements}


\bibliographystyle{spmpsci}      % mathematics and physical sciences
\bibliography{reference.bib}   % name your BibTeX data base

\begin{thebibliography}{10}
\providecommand{\url}[1]{{#1}}
\providecommand{\urlprefix}{URL }
\expandafter\ifx\csname urlstyle\endcsname\relax
  \providecommand{\doi}[1]{DOI~\discretionary{}{}{}#1}\else
  \providecommand{\doi}{DOI~\discretionary{}{}{}\begingroup
  \urlstyle{rm}\Url}\fi

\bibitem{al1985descent}
Al-Baali, M.: Descent property and global convergence of the
  {F}letcher-{R}eeves method with inexact line search.
\newblock IMA Journal of Numerical Analysis \textbf{5}(1), 121--124 (1985)

\bibitem{andrei2008another}
Andrei, N.: Another hybrid conjugate gradient algorithm for unconstrained
  optimization.
\newblock Numerical Algorithms \textbf{47}(2), 143--156 (2008)

\bibitem{armijo1966minimization}
Armijo, L.: Minimization of functions having {L}ipschitz continuous first
  partial derivatives.
\newblock Pacific Journal of mathematics \textbf{16}(1), 1--3 (1966)

\bibitem{beck2009fast}
Beck, A., Teboulle, M.: A fast iterative shrinkage-thresholding algorithm for
  linear inverse problems.
\newblock SIAM journal on imaging sciences \textbf{2}(1), 183--202 (2009)

\bibitem{Bertsekas99}
Bertsekas, D.: Nonlinear Programming.
\newblock Athena Scientific (1999)

\bibitem{boggess2015first}
Boggess, A., Narcowich, F.J.: A first course in wavelets with Fourier analysis.
\newblock John Wiley \& Sons (2015)

\bibitem{boydPCPE11admm}
Boyd, S., Parikh, N., Chu, E., Peleato, B., Eckstein, J.: Distributed
  optimization and statistical learning via the alternating direction method of
  multipliers.
\newblock Found. Trends Mach. Learn. \textbf{3}(1), 1--122 (2011)

\bibitem{candes2008enhancing}
Cand{\`e}s, E.J., Wakin, M.B., Boyd, S.P.: Enhancing sparsity by reweighted l1
  minimization.
\newblock J. Fourier Anal. Appl. \textbf{14}(5-6), 877--905 (2008)

\bibitem{chambolle1998nonlinear}
Chambolle, A., De~Vore, R.A., Lee, N.Y., Lucier, B.J.: Nonlinear wavelet image
  processing: variational problems, compression, and noise removal through
  wavelet shrinkage.
\newblock IEEE Transactions on Image Processing \textbf{7}(3), 319--335 (1998)

\bibitem{chan2014half}
Chan, R.H., Liang, H.X.: Half-quadratic algorithm for $\ell_p$-$\ell_q $
  problems with applications to tv-$\ell_1$ image restoration and compressive
  sensing.
\newblock In: Efficient algorithms for global optimization methods in computer
  vision, pp. 78--103. Springer (2014)

\bibitem{chen2010smoothing}
Chen, X., Zhou, W.: Smoothing nonlinear conjugate gradient method for image
  restoration using nonsmooth nonconvex minimization.
\newblock SIAM Journal on Imaging Sciences \textbf{3}(4), 765--790 (2010)

\bibitem{combettes2005signal}
Combettes, P.L., Wajs, V.R.: Signal recovery by proximal forward-backward
  splitting.
\newblock Multiscale Modeling \& Simulation \textbf{4}(4), 1168--1200 (2005)

\bibitem{dai1999nonlinear}
Dai, Y.H., Yuan, Y.: A nonlinear conjugate gradient method with a strong global
  convergence property.
\newblock SIAM Journal on optimization \textbf{10}(1), 177--182 (1999)

\bibitem{dai2001efficient}
Dai, Y.h., Yuan, Y.: An efficient hybrid conjugate gradient method for
  unconstrained optimization.
\newblock Annals of Operations Research \textbf{103}(1), 33--47 (2001)

\bibitem{daubechies2004iterative}
Daubechies, I., Defrise, M., De~Mol, C.: An iterative thresholding algorithm
  for linear inverse problems with a sparsity constraint.
\newblock Communications on Pure and Applied Mathematics: A Journal Issued by
  the Courant Institute of Mathematical Sciences \textbf{57}(11), 1413--1457
  (2004)

\bibitem{donoho2006compressed}
Donoho, D.L.: Compressed sensing.
\newblock IEEE Trans. Inf. Theory \textbf{52}(4), 1289--1306 (2006)

\bibitem{figueiredo2003algorithm}
Figueiredo, M.A., Nowak, R.D.: An {EM} algorithm for wavelet-based image
  restoration.
\newblock IEEE Transactions on Image Processing \textbf{12}(8), 906--916 (2003)

\bibitem{fletcher1964function}
Fletcher, R., Reeves, C.M.: Function minimization by conjugate gradients.
\newblock The computer journal \textbf{7}(2), 149--154 (1964)

\bibitem{gilbert1992global}
Gilbert, J.C., Nocedal, J.: Global convergence properties of conjugate gradient
  methods for optimization.
\newblock SIAM Journal on optimization \textbf{2}(1), 21--42 (1992)

\bibitem{giselsson2014monotonicity}
Giselsson, P., Boyd, S.: Monotonicity and restart in fast gradient methods.
\newblock In: 53rd IEEE Conference on Decision and Control, pp. 5058--5063.
  IEEE (2014)

\bibitem{golub1999inexact}
Golub, G.H., Ye, Q.: Inexact preconditioned conjugate gradient method with
  inner-outer iteration.
\newblock SIAM Journal on Scientific Computing \textbf{21}(4), 1305--1320
  (1999)

\bibitem{GuoLLtransform18}
Guo, L., Li, J., Liu, Y.: Stochastic collocation methods via minimisation of
  the transformed $l_1$-penalty.
\newblock East Asian Journal on Applied Mathematics \textbf{8}(3), 566--585
  (2018)

\bibitem{guo2021novel}
Guo, W., Lou, Y., Qin, J., Yan, M.: A novel regularization based on the error
  function for sparse recovery.
\newblock Journal of Scientific Computing \textbf{87}(1), 1--22 (2021)

\bibitem{hager2005new}
Hager, W.W., Zhang, H.: A new conjugate gradient method with guaranteed descent
  and an efficient line search.
\newblock SIAM Journal on optimization \textbf{16}(1), 170--192 (2005)

\bibitem{hager2006survey}
Hager, W.W., Zhang, H.: A survey of nonlinear conjugate gradient methods.
\newblock Pacific journal of Optimization \textbf{2}(1), 35--58 (2006)

\bibitem{hale2007fixed}
Hale, E.T., Yin, W., Zhang, Y.: A fixed-point continuation method for
  l1-regularized minimization with applications to compressed sensing.
\newblock CAAM TR07-07, Rice University \textbf{43}, 44 (2007)

\bibitem{hardt2014robustness}
Hardt, M.: Robustness versus acceleration (2014).
\newblock
  \urlprefix\url{http://blog.mrtz.org/2014/08/18/robustness-versus-acceleration.html}

\bibitem{hermey1999fitting}
Hermey, D., Watson, G.A.: Fitting data with errors in all variables using the
  huber m-estimator.
\newblock SIAM Journal on Scientific Computing \textbf{20}(4), 1276--1298
  (1999)

\bibitem{Hestenes&Stiefel:1952}
Hestenes, M.R., Stiefel, E.: Methods of conjugate gradients for solving linear
  systems.
\newblock Journal of research of the National Bureau of Standards \textbf{49},
  409--436 (1952)

\bibitem{hestenes1952methods}
Hestenes, M.R., Stiefel, E., et~al.: Methods of conjugate gradients for solving
  linear systems, vol.~49.
\newblock NBS Washington, DC (1952)

\bibitem{huang2017majorization}
Huang, G., Lanza, A., Morigi, S., Reichel, L., Sgallari, F.:
  Majorization--minimization generalized krylov subspace methods for
  $\ell_p$-$\ell_q$ optimization applied to image restoration.
\newblock BIT Numerical Mathematics \textbf{57}(2), 351--378 (2017)

\bibitem{huang2015nonconvex}
Huang, X.L., Shi, L., Yan, M.: Nonconvex sorted $\ell_1 $ minimization for
  sparse approximation.
\newblock Journal of the Operations Research Society of China \textbf{3}(2),
  207--229 (2015)

\bibitem{huber1987place}
Huber, P.J.: The place of the l1-norm in robust estimation.
\newblock Computational Statistics \& Data Analysis \textbf{5}(4), 255--262
  (1987)

\bibitem{lanzageneralized}
Lanza, A., Morigi, S., Reichel, L., Sgallari, F.: A generalized {K}rylov
  subspace method for $\ell_p$-$\ell_q$ minimization.
\newblock SIAM Journal on Scientific Computing \textbf{37}(5), S30--S50 (2015).
\newblock \doi{10.1137/140967982}.
\newblock \urlprefix\url{https://doi.org/10.1137/140967982}

\bibitem{li2015accelerated}
Li, H., Lin, Z.: Accelerated proximal gradient methods for nonconvex
  programming.
\newblock Advances in Neural Information Processing Systems \textbf{28},
  379--387 (2015)

\bibitem{liesen2013krylov2.3}
Liesen, J., Strakos, Z.: Mathematical characterisation of some {K}rylov
  subspace methods.
\newblock pp. 23--26. Oxford University Press (2013)

\bibitem{lorenz2011constructing}
Lorenz, D.A.: Constructing test instances for basis pursuit denoising.
\newblock IEEE Trans. Signal Process. \textbf{61}(5), 1210--1214 (2013)

\bibitem{louY18}
Lou, Y., Yan, M.: Fast l1-l2 minimization via a proximal operator.
\newblock J. Sci. Comput. \textbf{74}(2), 767--785 (2018)

\bibitem{louYHX14}
Lou, Y., Yin, P., He, Q., Xin, J.: Computing sparse representation in a highly
  coherent dictionary based on difference of $ {L_1} $ and $ {L_2 }$.
\newblock J. Sci. Comput. \textbf{64}(1), 178--196 (2015)

\bibitem{louYX16}
Lou, Y., Yin, P., Xin, J.: Point source super-resolution via non-convex l1
  based methods.
\newblock J. Sci. Comput. \textbf{68}, 1082--1100 (2016)

\bibitem{lu2014iterative}
Lu, Z.: Iterative reweighted minimization methods for $\ell_p$ regularized
  unconstrained nonlinear programming.
\newblock Mathematical Programming \textbf{147}(1), 277--307 (2014)

\bibitem{lv2009unified}
Lv, J., Fan, Y., et~al.: A unified approach to model selection and sparse
  recovery using regularized least squares.
\newblock Annals of Stat. \textbf{37}(6A), 3498--3528 (2009)

\bibitem{narushima2013smoothing}
Narushima, Y.: A smoothing conjugate gradient method for solving systems of
  nonsmooth equations.
\newblock Applied Mathematics and Computation \textbf{219}(16), 8646--8655
  (2013)

\bibitem{natarajan95}
Natarajan, B.K.: Sparse approximate solutions to linear systems.
\newblock SIAM J. Comput. \textbf{24}(2), 227--234 (1995)

\bibitem{nemirovski1985optimal}
Nemirovski, A.S., Nesterov, Y.E.: Optimal methods of smooth convex
  minimization.
\newblock Zhurnal Vychislitel'noi Matematiki i Matematicheskoi Fiziki
  \textbf{25}(3), 356--369 (1985)

\bibitem{nesterov1983method}
Nesterov, Y.: A method of solving a convex programming problem with convergence
  rate o (1/k2).
\newblock In: Soviet Mathematics Doklady, vol.~27, pp. 372--376 (1983)

\bibitem{nesterov2003introductory}
Nesterov, Y.: Introductory lectures on convex optimization: A basic course,
  vol.~87.
\newblock Springer Science \& Business Media (2003)

\bibitem{nocedal2006numerical}
Nocedal, J., Wright, S.: Numerical optimization.
\newblock Springer Science \& Business Media (2006)

\bibitem{pang2016smoothing}
Pang, D., Du, S., Ju, J.: The smoothing fletcher-reeves conjugate gradient
  method for solving finite minimax problems.
\newblock ScienceAsia \textbf{42}(1), 40--45 (2016)

\bibitem{parikh2014proximal}
Parikh, N., Boyd, S.: Proximal algorithms.
\newblock Found. Trends Opt. \textbf{1}(3), 127--239 (2014)

\bibitem{TA98}
Pham-Dinh, T., Le-Thi, H.A.: A {D.C.} optimization algorithm for solving the
  trust-region subproblem.
\newblock SIAM J. Optim. \textbf{8}(2), 476--505 (1998)

\bibitem{phamLe2005dc}
Pham-Dinh, T., Le-Thi, H.A.: The {DC} (difference of convex functions)
  programming and {DCA} revisited with {DC} models of real world nonconvex
  optimization problems.
\newblock Annals Oper. Res. \textbf{133}(1-4), 23--46 (2005)

\bibitem{polak1969note}
Polak, E., Ribiere, G.: Note sur la convergence de m{\'e}thodes de directions
  conjugu{\'e}es.
\newblock ESAIM: Mathematical Modelling and Numerical Analysis-Mod{\'e}lisation
  Math{\'e}matique et Analyse Num{\'e}rique \textbf{3}(R1), 35--43 (1969)

\bibitem{POLYAK19641}
Polyak, B.: Some methods of speeding up the convergence of iteration methods.
\newblock USSR Computational Mathematics and Mathematical Physics
  \textbf{4}(5), 1--17 (1964)

\bibitem{powell1977restart}
Powell, M.J.D.: Restart procedures for the conjugate gradient method.
\newblock Mathematical programming \textbf{12}(1), 241--254 (1977)

\bibitem{l1dl2}
Rahimi, Y., Wang, C., Dong, H., Lou, Y.: A scale invariant approach for sparse
  signal recovery.
\newblock SIAM J. Sci. Comput. \textbf{41}(6), A3649--A3672 (2019)

\bibitem{rivaie2015new}
Rivaie, M., Mamat, M., Abashar, A.: A new class of nonlinear conjugate gradient
  coefficients with exact and inexact line searches.
\newblock Applied Mathematics and Computation \textbf{268}, 1152--1163 (2015)

\bibitem{rockafellar2009variational}
Rockafellar, R.T., Wets, R.J.B.: Variational analysis, vol. 317.
\newblock Springer Science \& Business Media (2009)

\bibitem{roulet2020sharpness}
Roulet, V., d'Aspremont, A.: Sharpness, restart, and acceleration.
\newblock SIAM Journal on Optimization \textbf{30}(1), 262--289 (2020)

\bibitem{saad2003iterative}
Saad, Y.: Iterative methods for sparse linear systems.
\newblock SIAM (2003)

\bibitem{shen2012likelihood}
Shen, X., Pan, W., Zhu, Y.: Likelihood-based selection and sharp parameter
  estimation.
\newblock J. Am. Stat. Assoc. \textbf{107}(497), 223--232 (2012)

\bibitem{su2014differential}
Su, W., Boyd, S., Candes, E.: A differential equation for modeling nesterov’s
  accelerated gradient method: Theory and insights.
\newblock Advances in neural information processing systems \textbf{27},
  2510--2518 (2014)

\bibitem{sun2020adaptive}
Sun, Q., Zhou, W.X., Fan, J.: Adaptive huber regression.
\newblock Journal of the American Statistical Association \textbf{115}(529),
  254--265 (2020)

\bibitem{tong2000analysis}
Tong, C., Ye, Q.: Analysis of the finite precision bi-conjugate gradient
  algorithm for nonsymmetric linear systems.
\newblock Mathematics of computation \textbf{69}(232), 1559--1575 (2000)

\bibitem{Shannon}
Unser, M.: Sampling $-$ 50 years after shannon.
\newblock In: Proceedings of the IEEE, pp. 569 -- 587. IEEE (2000)

\bibitem{vonesch2007fast}
Vonesch, C., Unser, M.: A fast iterative thresholding algorithm for
  wavelet-regularized deconvolution.
\newblock In: Wavelets XII, vol. 6701, p. 67010D. International Society for
  Optics and Photonics (2007)

\bibitem{L1dL2_accelerated}
Wang, C., Yan, M., Rahimi, Y., Lou, Y.: Accelerated schemes for the $
  {L}_1/{L}_2 $ minimization.
\newblock IEEE Trans. Signal Process. \textbf{68}, 2660--2669 (2020)

\bibitem{fundamentals_Sec6.2}
Watkins, D.S.: Subspace iteration and simultaneous iteration.
\newblock pp. 420--428. John Wiley \& Sons (2010)

\bibitem{wright2009sparse}
Wright, S.J., Nowak, R.D., Figueiredo, M.A.: Sparse reconstruction by separable
  approximation.
\newblock IEEE Transactions on signal processing \textbf{57}(7), 2479--2493
  (2009)

\bibitem{wu2019signal}
Wu, C., Zhan, J., Lu, Y., Chen, J.S.: Signal reconstruction by conjugate
  gradient algorithm based on smoothing l1-norm.
\newblock Calcolo \textbf{56}(4), 1--26 (2019)

\bibitem{yinEX14}
Yin, P., Esser, E., Xin, J.: Ratio and difference of $l_1$ and $l_2$ norms and
  sparse representation with coherent dictionaries.
\newblock Comm. Inf. Syst. \textbf{14}(2), 87--109 (2014)

\bibitem{yinLHX14}
Yin, P., Lou, Y., He, Q., Xin, J.: Minimization of $\ell_{1-2}$ for compressed
  sensing.
\newblock SIAM J. Sci. Comput. \textbf{37}(1), A536--A563 (2015)

\bibitem{zhangX17}
Zhang, S., Xin, J.: Minimization of transformed ${L_1}$ penalty: Closed form
  representation and iterative thresholding algorithms.
\newblock Comm. Math. Sci. \textbf{15}, 511--537 (2017)

\bibitem{zhangX18}
Zhang, S., Xin, J.: Minimization of transformed ${L_1 }$ penalty: theory,
  difference of convex function algorithm, and robust application in compressed
  sensing.
\newblock Math. Program. \textbf{169}(1), 307--336 (2018)

\bibitem{zhang2009multi}
Zhang, T.: Multi-stage convex relaxation for learning with sparse
  regularization.
\newblock In: Adv. Neural Inf. Proces. Syst. (NIPS), pp. 1929--1936 (2009)

\end{thebibliography}

\end{document}